\documentclass[a4paper,11pt]{article}

\usepackage{amsmath,amssymb,amsfonts,mathrsfs,geometry,mathabx,xcolor}
\usepackage{comment}
\usepackage{enumerate,enumitem}
\usepackage{bm}
\usepackage{hyperref}
\usepackage{url}

\usepackage{amsthm}
\theoremstyle{definition}
\newtheorem{theorem}{Theorem}[section]
\newtheorem*{theorem*}{Theorem}
\newtheorem{definition}[theorem]{Definition}
\newtheorem*{definition*}{Definition}
\newtheorem{proposition}[theorem]{Proposition}
\newtheorem{observation}[theorem]{Observation}
\newtheorem*{proposition*}{Proposition}
\newtheorem{lemma}[theorem]{Lemma}
\newtheorem*{lemma*}{Lemma}
\newtheorem{corollary}[theorem]{Corollary}
\newtheorem*{corollary*}{Corollary}

\newtheorem{example}[theorem]{Example}
\newtheorem*{example*}{Example}
\newtheorem{remark}[theorem]{Remark}
\newtheorem*{remark*}{Remark}

\newtheorem*{notation*}{Notation}

\numberwithin{equation}{section} 

\newcommand{\N}{\mathbb{N}}

\newcommand{\C}{\mathbb{C}}
\newcommand{\G}{\mathbb{G}}

\DeclareMathOperator{\Tr}{Tr}
\DeclareMathOperator{\tr}{tr}
\DeclareMathOperator{\diag}{diag}
\providecommand{\abs}[1]{\lvert#1\rvert}

\newcommand{\E}{\mathbb{E}}
\renewcommand{\phi}{\varphi}
\DeclareMathOperator{\NC}{NC}

\DeclareMathOperator{\M}{M}
\DeclareMathOperator{\fc}{\kappa}
\newcommand{\ncmu}{\mu}

\DeclareMathOperator{\Kr}{Kr}
\renewcommand{\epsilon}{\varepsilon} 

\newcommand{\A}{\mathcal{A}}
\newcommand{\B}{\mathcal{B}}
\newcommand{\F}{\mathcal{F}}

\newcommand{\V}{\mathcal{V}}
\renewcommand{\i}{\mathbf{i}}

\newcommand{\I}{I}

\newcommand{\tildephi}{\widetilde{\varphi}}
\newcommand{\tildekappa}{\widetilde{\kappa}}

\usepackage{stmaryrd}

\newcommand{\Alg}[2]{\mathop{#1 \langle#2\rangle}}
\newcommand{\nuAlg}[2]{\mathop{#1 \langle#2\rangle_0}}
\newcommand{\AAlg}[1]{\mathop{\langle#1\rangle}}
\newcommand{\orderedprod}{\mathop{\overrightarrow{\prod}}}  

\newcommand{\subproof}[1]{\vspace{2mm}\noindent\underline{#1}}

\newcommand{\cfconv}[4]{\mathop{{#1}\hspace{0.1em}{}_{#2}\hspace{-0.3em}\boxplus_{#3}\hspace{-0.1em}{#4} }}  
\newcommand{\ifconv}[4]{\mathop{{#1}\hspace{0.1em}{}_{#2}\hspace{-0.1em}{\boxplus\hspace{0.3mm}'}_{\hspace{-0.8mm}#3}\hspace{0.1em}{#4} }}  
\newcommand{\Restr}[2]{{#1}{\restriction}_{#2}} 
\newcommand\clhd[2]{\mathop{{#1} \lhd_c {#2}} }   

\newcommand{\X}{X}    
\newcommand{\Per}{G}  
\newcommand{\per}{g}   

\newcommand{\ukappa}{\underline{\kappa}}   

\title{Free probability of type B prime}
\author{Katsunori Fujie and Takahiro Hasebe}
\date{\today}

\begin{document}

\maketitle

\begin{abstract}
  Free probability of type B was invented by Biane--Goodman--Nica, and then it was generalized by Belinschi--Shlyakhtenko and F\'evrier--Nica to infinitesimal free probability. The latter found its applications to eigenvalues of perturbed random matrices in the work of Shlyakhtenko and C\'{e}bron--Dahlqvist--Gabriel.
  This paper offers a new framework, called ``free probability of type B${}^\prime$\,'', which appears in the large size limit of independent unitarily invariant random matrices with perturbations.
  Our framework is related to Boolean, free, (anti)monotone, cyclic-(anti)monotone and conditionally free independences.
  We then apply the new framework to the principal minor of unitarily invariant random matrices, which leads to the definition of a multivariate inverse Markov--Krein transform and asymptotic infinitesimal freeness of principal minors.
\end{abstract}

\tableofcontents

\section{Introduction}

\subsection{Backgrounds}

Voiculescu initiated free probability theory in the 1980s to attack the famous free group factor isomorphism problem in operator algebras.
This theory is an analogue of probability theory in which a probability space is replaced by a noncommutative probability space (\textit{ncps} for short) and independence is replaced by free independence (or freeness), see Subsection \ref{sec:free}. 
Intensive work on free probability led to numerous connections to other fields, e.g., asymptotic freeness in random matrix theory \cite{CS06,Voi91} and theory of cumulants governed by combinatorics of noncrossing partitions \cite{Spe94}. 
Note that noncrossing partitions may be called of type A because it can be somehow associated with symmetric groups, the Coxeter groups of type A \cite{Bia97}. 

On the other hand, in 1997 Reiner introduced noncrossing partitions of type B associated with hyperoctahedral groups, the Coxeter groups of type B \cite{R}.
Motivated by this work, in 2003, Biane et al.\ constructed a theory called \textit{free probability of type B} \cite{BGN}. 
This was a new framework of noncommutative probability. They introduced type B free cumulants and a notion of independence and then proved that the vanishing of mixed cumulants is equivalent to independence. Combinatorics of free cumulants of type B are governed by noncrossing partitions of type B. 
When contrasting the difference, we call the standard framework of free probability that of \textit{type A} in this paper. 
Free probability of type B has then attracted attention. Belinschi and Shlyakhtenko introduced the notion of \textit{infinitesimal ncps} and \textit{infinitesimal freeness} and studied its relation to type B free probability \cite{BS}. Its analytic aspect was intensively studied, where free subordination functions played a prominent role to describe type B free convolution. 
Then F\'evrier and Nica developed infinitesimal freeness from combinatorial aspects \cite{FN}: they defined the infinitesimal cumulants and discovered the vanishing of them is equivalent to infinitesimal freeness; 
they associated the infinitesimal ncps with a ncps of type B, in which situation infinitesimal freeness and type B freeness are equivalent. 
Note that the terms ``type B'' and ``infinitesimal'' are sometimes used for the same meaning due to the circumstances described above, so attention should be paid to the context. In the present paper, distinguishing these terms is crucial.

Infinitesimal free probability found its applications to random matrices. We briefly overview known results. Random matrices perturbed by finite rank matrices are known to show interesting behaviors: such a random matrix may have a finite number of eigenvalues (``outliers'') away from the other eigenvalues. Outliers typically appear as soon as the perturbation is bigger than a certain threshold.
This phenomenon is called the BBP phase transition, named after Baik, Ben Arous, P\'{e}ch\'{e} because of their seminal paper \cite{BBP}. Finite rank perturbation of unitarily invariant random matrices was studied by Benaych-Georges and Nadakuditi \cite{BN}, whose results were later extended to the case of sums of spiked models by Belinschi et al.\ \cite{BBCF}.
Shlyakhtenko found a connection between infinitesimal freeness and random matrices with finite rank perturbation and gave an explanation of the appearance of outliers of random matrices \cite{S}. 
C\'ebron, Dahlqvist and Gabriel \cite{CDG} proved a surprising relation between outliers of random matrices and conditional freeness of Bo\.zejko, Leinert and Speicher \cite{BLS,BS1} by using vector states. 
In 2018, Collins, Hasebe, and Sakuma gave an abstract framework to Shlyakhtenko's aforementioned work and named the calculation rule of moments {\it cyclic-monotone independence} because of its resemblance with monotone independence \cite{CHS}. Not only being similar, cyclic-monotone independence is directly connected to monotone independence: Arizmendi, Hasebe and Lehner proved that the canonical operator model for monotone independence on the tensor product Hilbert space satisfies cyclic-monotone independence with respect to the vacuum state and the trace \cite{AHL}; C\'ebron, Dahlqvist, and Gabriel showed how monotone independence appears from cyclic-monotone independence in an abstract setting \cite{CDG};
Collins, Leid and Sakuma constructed matrix models for monotone independence and cyclic-monotone independence \cite{CLS}. 

Although work on outliers of perturbed models is mainly devoted to sums and multiplications, some results on polynomials on unitarily invariant random matrices with perturbation are obtained by Belinschi, Bercovici and Capitaine \cite{BBC}. 
For models having only discrete spectra, Arizmendi and Celestino gave an algorithm for computing eigenvalues of polynomials on asymptotically cyclic-monotone independent random matrices \cite{AC}, and Collins et al.\ gave an algorithm for computing fluctuations of eigenvalues of polynomials on \textit{trivially independent}\footnote{See Definition \ref{def:trivial} of the present paper.} random matrices \cite{CFHLS}.

The purpose of the present paper is to offer a better understanding of the developments mentioned above.

\subsection{Overview of main results and structure of the paper}

In this paper, we provide a comprehensive framework ``\textit{type B\,${}^\prime$ ncps}'' to describe random matrices with perturbation. This framework arises from the random matrix models 
\begin{align}
  \{U_i A_i U_i^* + F_i\}_{i=1,2} \quad \text{and} \label{eq:model}\\
  \{U_i A_i U_i^* + V_i F_i V_i^*\}_{i=1,2}\label{eq:model2}, 
\end{align}
where $A_i= A_i^{(N)}$ and $F_i=F_i^{(N)}$ are deterministic matrices in $\M_N(\C)$, $F_i$ are of uniformly bounded ranks and $\{U_1,U_2,V_1,V_2\}$ are independent Haar unitary matrices.\footnote{The setting here is simplified so that a core idea can be easily grasped; results in later sections will be proved in a more general setting, e.g., the index is more general and the ranks of $F_i$ are not necessarily uniformly bounded.} We assume that, in the large $N$ limit, the ``main parts'' $U_i A_i U_i^*$ converge in distribution (with respect to the normalized trace) to elements $a_i$ in a unital algebra $\A$, and the ``perturbation parts'' $F_i$ converge in distribution (with respect to the nonnormalized trace) to elements $f_i$ in an $\A$-algebra $\F$.
The sums $U_i A_i U_i^* + F_i$ then converge (in a suitable sense) to $a_i+f_i \in \A \oplus \F$, which we also denote by $(a_i,f_i)$ to distinguish the main part and perturbation. The relationship between $a_i$ and $f_i$ is described in a way quite similar to type B free probability (in the sense of Biane, Goodman and Nica \cite{BGN}) but slightly different. 
More precisely, in both frameworks of type B and of type B${}^\prime$, a random variable consists of its main part and perturbation part, denoted as a pair $(a,f) \in \A \oplus \F$ as above. The difference appears in the multiplication rule. In the type B setting, the perturbation part $f$ is considered to be ``infinitesimal'', so that multiplication is defined by
\begin{equation*}
  (a_1, f_1) \cdot (a_2, f_2) := (a_1a_2, a_1 f_2 + f_1 a_2).
\end{equation*}
By contrast, in our new framework of type B${}^\prime$, the second component $f$ represents a finite rank matrix, so that the multiplication rule is 
\begin{equation*}
  (a_1, f_1) (a_2, f_2) := (a_1a_2, a_1 f_2 + f_1 a_2+f_1f_2), 
\end{equation*}
i.e., the product of perturbations is not neglected. This is natural because the product $F_1F_2$ of finite rank matrices $F_1,F_2$ might still be comparable with $A_1F_2$ and $F_1A_2$ in the large $N$ limit, e.g., in case $F_1=F_2$ is a rank one projection and $A_1$ and $A_2$ are the identity matrix. 
See Subsections \ref{sec:ncps} and \ref{sec:typeB_NCPS} for further details.

The asymptotic behavior of the models \eqref{eq:model} and \eqref{eq:model2} naturally yields two respective notions of independence ``weak B${}^\prime$-freeness'' and ``B${}^\prime$-freeness'' for abstract random variables $(a_1,f_1)$ and $(a_2,f_2)$. The former notion stands for the situation where $a_1$ and $a_2$ are free and $(\langle a_1,a_2\rangle, \langle f_1,f_2\rangle)$ is cyclic-antimonotone.\footnote{We mostly use the term ``cyclic-antimonotone'' instead of ``cyclic-monotone'', see the texts following Definition \ref{def:cm} for the reason.}
The latter notion, B${}^\prime$-freeness, additionally requires that the product $f_1f_2$ behaves like the zero element (in which case we call $f_1$ and $f_2$ to be trivially independent), but does not require $f_i^n$ ($n \ge 2$) to behave like zero. The asymptotic trivial independence of $\{V_i F_i V_i^*\}_{i=1,2}$ was already pointed out in \cite[Proposition 4.11]{CHS}. See Subsection \ref{sec:B'free} for detailed definitions of the two notions of independence and see Subsection \ref{sec:RM} for asymptotic independence of our random matrix models \eqref{eq:model} and \eqref{eq:model2}.

The framework of type B${}^\prime$ is intimately connected to infinitesimal freeness and infinitesimally free cumulants. 
In Section \ref{sec:BFIF} we prove that B${}^\prime$-freeness and its weak version can be characterized in terms of infinitesimal freeness (Theorem \ref{thm:B'free}). 

Weak B${}^\prime$-freeness turns out to be connected to conditional freeness of Bo\.zejko, Leinert and Speicher \cite{BLS,BS1}. More precisely, in Section \ref{sec:cfree} we prove that if $(a_1,f_1)$ and $(a_2,f_2)$ are weakly B${}^\prime$-free, then the conditional freeness of $(a_1,f_1)$ and $(a_2,f_2)$ with respect to suitable linear functionals is equivalent to the Boolean independence (instead of trivial independence) of $f_1$ and $f_2$ (Theorem \ref{thm:c-free}). Note that several different connections of infinitesimal freeness and conditional freeness are pointed out in the literature \cite{BS,CDG,FMNS}.

As an application of type B${}^\prime$ free probability, in Section \ref{sec:random_matrix} we study the principal submatrix of unitarily invariant random matrices.
F\'evrier and Nica's infinitesimal analysis of free compressions fits very well in this model. The principal submatrix of $A \in \M_N(\C)$ can be written as $P A P$, where $P=\diag(1,1,\dots,1,0)$. The projection $P$ is almost the identity matrix in the large $N$ limit, but in the infinitesimal framework we can distinguish $P$ from the identity matrix in the large $N$ limit.
It is quite interesting that the inverse Markov--Krein transform appears in the calculations of infinitesimal distributions.\footnote{The Markov--Krein transform also appears in other contexts of random matrices, see \cite{AGVP,MP}.
For more applications of Markov--Krein transform, the reader is referred to \cite{FH,K} and the references therein.} This was already observed in \cite{FH} for the single variate case. In the present paper, we propose a notion of a ``multivariate inverse Markov--Krein transform.''
We also prove an asymptotic infinitesimal freeness for the principal minors, which shows our type B${}^\prime$ setting is appropriate for understanding principal minors. 

In the rest of this section, we introduce preliminary materials required in this paper.

\subsection{Freeness}\label{sec:free}
Throughout this paper, we assume that $(\A, \varphi)$ is a \textit{noncommutative probability space} (abbreviated to \textit{ncps} as already mentioned): $\A$ is an algebra over $\C$ with unit $1_\A$ and $\varphi$ is a unital linear functional on $\A$.
It is common to assume additional structures on $(\A,\varphi)$, e.g.,\ $\A$ is a $\ast$-algebra or $C^*$-algebra, $\varphi$ is a state or $\varphi$ is tracial; however, a large part of this paper does not require these additional structures. 

Before going into discussions, we introduce technical notations which are useful for simplifying various descriptions.
For an index set $I$ and natural numbers $n \in \N$, we define the sets of \textit{alternating sequences}
\begin{align*}
&I^{(n)} := \{ (i_1, i_2,\dots, i_n) \in I^n \mid i_1\neq i_2 \ne \cdots \neq i_n\}~(n\ge2), \quad I^{(1)}:=I 
\quad \text{and} \quad \\
&I^{(\infty)} := \bigcup_{n\in\N} I^{(n)}, 
\end{align*} 
where $i_1\ne i_2 \ne \cdots \ne i_n$ denotes the situation that the neighboring indices are distinct, i.e.,\ $i_k \ne i_{k+1}$ for all $1 \le k \le n-1$. 
Let $(\A_i)_{i \in I}$ be a family of subalgebras of $\A$ containing $1_\A$.
Every subalgebra $\A_i$ has the canonical decomposition $\A_i = \mathring{\A}_i \oplus \C 1_\A$ (as vector spaces) given by $a = \mathring{a} + \phi(a)1_\A$ for $a \in \A_i$, where $\mathring{\A}_i := \{a \in \A_i \mid \phi(a)=0 \}$ denotes the set of \textit{centered elements}. 
We introduce the abbreviations 
\[
\A_{\i} := \A_{i_1} \times \A_{i_2} \times \cdots \times \A_{i_n} \qquad \text{and} \qquad \mathring{\A}_{\i} := \mathring{\A}_{i_1} \times \mathring{A}_{i_2} \times \cdots \times \mathring{\A}_{i_n}
\]
for $\i = (i_1, i_2,\dots, i_n) \in I^{(\infty)}$.

\begin{definition}[{\cite{V85}}] \label{def:freeness} A family of subalgebras $(\A_i)_{i \in I}$ of $\A$ containing $1_\A$ are said to be \textit{freely independent} (or \textit{free}) in $(\A, \varphi)$ if $\varphi(a_1a_2 \cdots a_n)=0$ holds for
  any $\i = (i_1,i_2, \dots, i_n) \in \I^{(\infty)}$ 
  and any $(a_1, a_2,\dots, a_n)$ 
  $ \in \mathring{\A}_\i$. 
 \end{definition} 

\subsection{Boolean, monotone, antimonotone and c-free independences}

\begin{definition}[{\cite{SW97}}]
 Subalgebras $(\A_i)_{i \in I}$ of $\A$ (not necessarily containing $1_\A$) are called \textit{Boolean independent} if  
 \begin{enumerate}[label=\rm(B),leftmargin=1cm]
     \item  $\phi(a_1a_2 \cdots a_n)= \phi(a_1) \phi(a_2)\cdots \phi(a_n)$ whenever $\i=(i_1,i_2,\dots, i_n)\in \I^{(\infty)}$ and $(a_1,a_2,\dots,a_n)\in \A_{\i}$. 
 \end{enumerate}
\end{definition}

It immediately follows from the definition that Boolean independence demands that $\phi$ is nontracial and each subalgebra $\A_i$ does not include $1_\A$ unless in trivial situations.

\begin{definition}[{\cite{Mur00}}]\label{defi:monotone}
A pair of subalgebras $(\A_1, \A_2)$ of $\A$ (not necessarily containing $1_\A$) is called \textit{monotonically independent} if 
\begin{enumerate}[label=\rm(M),leftmargin=1cm]
    \item  $\phi(y_0 x_1 y_1 x_2 y_2  \cdots  x_n y_n)= \phi(x_1 x_2 \cdots x_n) \phi(y_0) \phi(y_1)\cdots \phi(y_n)$ for all $x_1,x_2,\dots, x_n \in \A_1$ and $  y_0, y_1,\dots,y_n  \in \langle \A_2, 1_\A\rangle. $  
\end{enumerate}
In this situation we also say that the  reversed pair $(\A_2, \A_1)$ is  \textit{antimonotonically independent}. 
\end{definition}

Consider another unital linear functional $\psi$ on $\A$.
A generalization of freeness and Boolean independence was introduced in \cite{BS1}  referring to $(\A, \phi, \psi)$.

\begin{definition}[{\cite{BLS,BS1}}]
 Subalgebras $(\A_i)_{i \in I}$ of $\A$ containing $1_\A$ are called \textit{conditionally free} (often abbreviated to \textit{c-free}) if they are free with respect to $\phi$ and 
 \begin{enumerate}[label=\rm(CF),leftmargin=1.2cm]
     \item  $\psi(a_1a_2 \cdots a_n)= \psi(a_1) \psi(a_2)\cdots \psi(a_n)$ whenever $n \in \N$, $(i_1,i_2,\dots, i_n)\in \I^{(n)}$ and $a_j \in \A_{i_j}$ such that $\varphi(a_j)=0$ ($1 \le j \le n$). 
 \end{enumerate}

\end{definition}

Note that if $(\A_i)_{i \in I}$ are c-free then $\Restr{\psi}{\langle \A_i : i\in I \rangle}$ is uniquely determined by $\{\Restr{\phi}{\A_i}\mid i\in I\}$ and $\{\Restr{\psi}{\A_i}\mid i\in I\}$.

\subsection{Infinitesimal freeness} 
\label{sec:infinitesimal}

\begin{definition}[\cite{FN}]
  Suppose that $\B$ is a unital algebra over $\C$ and $\phi, \phi' \colon \B \to \C$ are linear maps with $\phi(1_\B) = 1$, $\phi'(1_\B)=0$.
  We call $(\B, \phi, \phi')$ an \textit{infinitesimal ncps}. For each subalgebra $\B_1$ of $\B$ we denote by $\mathring\B_1$ the subspace $\{b \in \B_1: \varphi(b)=0\}$. 
\end{definition}

\begin{definition}[{\cite[Definition 1.1]{FN}}]\label{def:infinitesimal_free}
  Let $(\B, \phi, \phi')$ be an infinitesimal ncps and let $(\B_i)_{i\in I}$ be subalgebras of $\B$ containing the unit of $\B$.  We will say that $(\B_i)_{i\in I}$ are \textit{infinitesimally free} with respect to $(\phi, \phi')$ if the following condition is satisfied:
 \begin{enumerate}[label=\rm(IF),leftmargin=1cm] 
  \item\label{item:infinitesimal_free} 
    For every $\i =(i_1,i_2,\dots, i_n) \in \I^{(\infty)}$ and $(b_1,b_2,\dots, b_n) \in \mathring\B_\i$ we have
 $\phi(b_1b_2 \cdots b_n)=0$ and
    \begin{equation} \label{eq:infinitesimal_free}
      \phi'(b_1b_2 \cdots b_n) = 
      \begin{cases}
        \phi(b_1 b_n) \phi(b_2 b_{n-1}) \cdots \phi(b_{(n-1)/2}b_{(n+3)/2}) \phi'(b_{(n+1)/2}) & \\
         \qquad\quad \text{if $n$ is odd and $i_1 = i_n$, $i_2 = i_{n-1}$, \dots, $i_{(n-1)/2} = i_{(n+3)/2}$},   \\
        0, \qquad \text{otherwise}.
      \end{cases}
    \end{equation}
  \end{enumerate}
\end{definition}

Biane, Goodman, Nica's ncps of type B and the present paper's new framework ncps of type B$'$ will provide examples of infinitesimal ncps, see Subsection \ref{sec:typeB_NCPS}. 
Moreover, the freeness of type B is equivalent to infinitesimal freeness \cite[Formula (2.2)]{FN}, and also the freeness of type B$'$ is closely related to infinitesimal freeness, see Theorem \ref{thm:B'free}.

\subsection{Free cumulants and infinitesimal free cumulants}
\label{subsec:FreeCumulants}

Let $\NC(n)$ denote the set of noncrossing partitions of $[n] := \{ 1, 2,\dots, n \}$ and $\ncmu_n$ the M\"{o}bius function on the poset $\NC(n)$.   The maximum and minimum elements are denoted by  $1_n=\{[n]\}$ and $0_n=\{\{k\}\mid k \in [n]\}$, respectively. More generally, we sometimes consider the set of noncrossing partitions $\NC(L)$ on a finite linearly ordered set $L$, which is defined in the obvious way. 

Let $(\B,\phi,\phi')$ be an infinitesimal ncps. 
Let $\G$ be the algebra $\C[\epsilon]/(\epsilon^2)$ where $\epsilon$ is an indeterminate and let $\tildephi\colon \B \to \G$ be defined as follows:
\begin{equation*}
  \tildephi(b) = \varphi(b) + \epsilon\phi'(b), \qquad b \in \B.
\end{equation*}
The standard theory of free cumulants (see e.g.\ \cite{NS}) works for $\G$-valued ncps $(\B,\tildephi)$ with the same arguments. Then one can define infinitesimal free cumulants by looking at the coefficient of $\epsilon^1$ of the $\G$-valued free cumulants. For  the reader's convenience, an overview is provided below.

\subsubsection*{Definition of cumulants}

We define the multilinear functionals $\tildephi_n \colon \B^n \to \G$ by 
\[
\tildephi_n[b_1,b_2, \dots, b_n] := \tildephi(b_1 b_2\cdots b_n)
\]
for $n \in \N$, $b_i \in \B$ ($1 \le i \le n$). 
Then the family $(\tildephi_n)_{n \in \N}$ is multiplicatively extended to $(\tildephi_\pi\colon \B^n \to \G)_{n \in \N,\; \pi \in \NC(n)}$ as follows:
\begin{equation} \label{eq:multiplicative}
  \tildephi_\pi := \prod_{V \in \pi} \Restr{\tildephi}{V},
\qquad \text{where} \qquad  
  \Restr{\tildephi}{V} [b_1,b_2, \dots, b_n] := \tildephi_k [b_{i_1},b_{i_2}, \dots, b_{i_k}]
\end{equation}
for a nonempty subset $V = \{ i_1, i_2, \dots, i_k\} \subseteq [n]$ with $i_1<i_2< \cdots < i_k$.
The coefficients $\epsilon^0$ and $\epsilon^1$ of $\widetilde \phi$ are denoted by $(\varphi_\pi\colon \B^n \to \C)$ and $(\varphi'_\pi\colon \B^n \to \C)$, respectively: 
\begin{equation} \label{eq:tildephi1}
  \tildephi_\pi = \varphi_\pi + \epsilon \varphi'_\pi, \qquad \pi \in \NC(n). 
\end{equation}
This definition leads to the multiplicativity 
  \begin{equation*}
    \varphi_\pi = \prod_{V \in \pi} \Restr{\varphi}{V},\qquad \pi \in \NC(n) 
  \end{equation*} 
 that appears in type A free probability.
 By contrast, the formula for $\varphi'_\pi$ is very different and is given by 
  \begin{equation}
    \varphi'_\pi = \sum_{V \in \pi} \Restr{\varphi'}{V} \prod_{W \in \pi\setminus\{V\}} \Restr{\varphi}{W}. \label{eq:varphi'}
  \end{equation}
The prime operation thus behaves like differentiation. 
Since $\pi \setminus \{V\}$ can be regarded as a noncrossing partition on the totally ordered set $[n] \setminus V$,  formula \eqref{eq:varphi'} can also be written in the concise form 
\[
\varphi'_\pi = \sum_{V \in \pi} \Restr{\varphi'}{V}  \varphi_{\pi \setminus \{V\}}. 
\]

Let $\ncmu_n$ be the M\"{o}bius function on $\NC(n)$. The \textit{$\G$-valued free cumulants} $(\tildekappa_\pi\colon \B^n\to\G)_{\pi \in \NC(n), n \in \N}$ are defined by 
\begin{equation} \label{cumulant-moment_tilde}
  \tildekappa_\pi := \sum_{\substack{\sigma \in \NC(n) \\ \sigma \le \pi}} \tildephi_\sigma \ncmu_n(\sigma, \pi), \qquad \pi \in \NC(n). 
\end{equation}
The multilinear functionals $(\kappa_\pi\colon \B^n \to \C)_{\pi \in \NC(n), n \in \N}$ and $(\kappa'_\pi\colon \B^n \to \C)_{\pi \in \NC(n), n \in \N}$, defined by 
\begin{equation} \label{eq:tildekappa1}
  \tildekappa_\pi = \kappa_\pi + \epsilon \kappa'_\pi,  
\end{equation}
are respectively called the \textit{free cumulants} and \textit{infinitesimal free cumulants}.

\subsubsection*{Moment-cumulant formulas}
Imitating the proofs of the $\C$-valued case, we obtain the formulas 
\begin{align}
  \tildephi_\pi   & = \sum_{\sigma \le \pi} \tildekappa_\sigma \quad \text{and}\label{moment-cumulant_tilde} \\
  \tildekappa_\pi & = \prod_{V \in \pi} \Restr{\tildekappa}{V} \label{eq:tildekappa2}
\end{align}
 for all $\pi \in \NC(n)$.
Equations \eqref{cumulant-moment_tilde} and \eqref{moment-cumulant_tilde} are called the ($\G$-valued) \textit{moment-cumulant formulas}.
Comparing the coefficients of $\epsilon^0$ and $\epsilon^1$ in formulas \eqref{cumulant-moment_tilde}, \eqref{moment-cumulant_tilde} and \eqref{eq:tildekappa2} yields the formulas
\begin{align}
 &\kappa_\pi   = \sum_{\sigma \le \pi} \varphi_\sigma \ncmu_n(\sigma, \pi) \qquad \text{and} \qquad 
  \varphi_\pi  = \sum_{\sigma \le \pi} \kappa_\sigma,                             \label{eq:moment-_cumulant_A} \\
  &\kappa'_\pi   = \sum_{\sigma \le \pi} \varphi'_\sigma \ncmu_n(\sigma, \pi) \qquad \text{and} \qquad   
  \varphi'_\pi  = \sum_{\sigma \le \pi} \kappa'_\sigma,     \notag                        \\
  &\kappa_\pi  = \prod_{V \in \pi} \Restr{\kappa}{V} \qquad \text{and} \qquad \kappa'_\pi   = \sum_{V \in \pi} \Restr{\kappa'}{V} \prod_{W \in \pi, W \neq V} \Restr{\kappa}{W}. \label{eq:multiplicativity_kappa}
\end{align}

\subsection{Distributions and convolutions}
Let $(\B,\phi,\phi')$ be an infinitesimal ncps. 

A linear functional $\mu \colon \C[x]\to \C$ is called a \textit{distribution}. 
If $\mu,\nu \colon \C[x]\to \C$ are linear functionals with $\mu(1)=1$ and $\nu(1)=0$ then the pair $(\mu,\nu)$ is called an \textit{infinitesimal distribution}.

For $b \in \B$ the linear functional $\mu_b \colon \C[x]\to \C$ defined by $\mu_b(x^n):=\phi(b^n), n\in \N\cup\{0\}$ is called \textit{the distribution of $b$ with respect to $\phi$.} Moreover, together with the linear functional $\mu_b'\colon \C[x]\to \C$ defined by $\mu_b'(x^n)= \phi'(x^n), x\in \N\cup\{0\}$, the pair $(\mu_b,\mu_b')$ is called the \textit{infinitesimal distribution of $b$ with respect to $(\phi, \phi')$}. 

Suppose that $b_1,b_2\in \B$ have distributions $\mu_1,\mu_2$ with respect to $\phi$, respectively. 
If $b_1, b_2\in\B$ are free, monotone, antimonotone and Boolean independent with respect to $\phi$, then the distribution of $b_1 +b_2$ is called the \textit{free, monotone, antimonotone and Boolean convolution} of $\mu_1, \mu_2$ and are denoted by $\mu_1 \boxplus \mu_2$, $\mu_1 \rhd \mu_2$, $\mu_1 \lhd \mu_2$ and $\mu_1 \uplus \mu_2$, respectively. Note that $\mu_1 \rhd \mu_2 = \mu_2 \lhd \mu_1$ by definition. 

If $b_1,b_2\in \B$ are infinitesimally free with respect to $(\phi, \phi')$ having infinitesimal distributions $(\mu_1,\mu'_1), (\mu_2,\mu'_2)$, respectively, then the distribution of $b_1+b_2$ with respect to $\phi'$ is called the \textit{infinitesimal free convolution} and is denoted by $\ifconv{\mu'_1}{\mu_1}{\mu_2}{\mu'_2}$. We also call the binary operation $(\mu_1, \mu'_1) \boxplus (\mu_2,\mu'_2):=(\mu_1\boxplus \mu_2, \ifconv{\mu'_1}{\mu_1}{\mu_2}{\mu'_2})$ the infinitesimal free convolution. 

Likewise, let $\psi$ be a unital linear functional on $\B$ and $\nu_1, \nu_2$ be the distributions of $b_1,b_2$ with respect to $\psi$, respectively. If $b_1,b_2\in \B$ are c-free with respect to $(\phi, \psi)$ then the distribution of $b_1+b_2$ with respect to $\psi$ is called the \textit{c-free convolution} and is denoted by $\cfconv{\nu_1}{\mu_1}{\mu_2}{\nu_2}$. In the above two cases the distribution of $b_1+b_2$ with respect to $\phi$ is the free convolution $\mu_1 \boxplus \mu_2$.

In this paper we also work with another convolution called cyclic-antimonotone convolution. It will be defined later in Corollary \ref{cor:cm} since we need some more preparation.

For computing convolutions, useful machineries are the Cauchy transform $G_\mu$ and its reciprocal $F_{\mu}$ of a distribution $\mu$: 
\[
G_\mu(z) = \sum_{n=0}^\infty \frac{\mu(x^n)}{z^{n+1}} \qquad \text{and} \qquad F_\mu(z) = \frac{1}{G_\mu(z)}, 
\]
both defined as formal Laurent series. The formal compositional inverse of $F_\mu(z)$ is denoted by $F_\mu^{-1}(z)$. 
With the above notation, free, Boolean and monotone convolutions are characterized by 
\begin{align}
F_{\mu_1 \boxplus \mu_2}^{-1}(z) &= F_{\mu_1}^{-1}(z)+F_{\mu_2}^{-1}(z)-z, & \cite[\rm Corollary~5.8]{BV93} \notag \\
F_{\mu_1 \uplus \mu_2}(z) &= F_{\mu_1}(z)+F_{\mu_2}(z)-z, & \cite{SW97} \label{eq:Boolean_convolution} \\
F_{\mu_1 \rhd \mu_2}(z) &= F_{\mu_1}\circ F_{\mu_2}(z). & \cite[\rm Theorem~3.1]{Mur00} \label{eq:monotone_convolution}
\end{align}
According to \cite[Proposition 5.2]{Has} or \cite[Proposition 3]{Bel} the following formula holds: 
\begin{equation} \label{eq:c-free_convolution}
F_{\cfconv{\nu_1\!}{\mu_1\,\,}{\mu_2}{\nu_2}} = F_{\nu_1}\circ \omega_1 + F_{\nu_2}\circ \omega_2  - F_{\mu_1 \boxplus \mu_2},  
\end{equation}
where $\omega_i:= F_{\mu_i}^{-1} \circ F_{\mu_1 \boxplus \mu_2}$. 
Finally, infinitesimal free convolution is characterized by 
\begin{equation}\label{eq:infinitesimal_convolution}
  G_{\ifconv{\nu_1}{\mu_1}{\mu_2}{\nu_2}}(z) = G_{\nu_1}(\omega_1(z)) \omega_1'(z) +  G_{\nu_2}(\omega_2(z)) \omega_2'(z).  \qquad \cite[\rm Proposition~20]{BS} 
\end{equation}


\section{Framework and random matrix models}\label{sec:setup}

\subsection{Ncps of type B${}^\prime$, trivial independence and cyclic-antimonotone independence} \label{sec:ncps}

\begin{definition} \begin{enumerate}[label=\rm(\roman*),leftmargin=1cm]
\item   Let $\F$ be a $\C$-algebra and $\Phi\colon \F \to \C$ be a linear functional.
  The pair $(\F, \Phi)$ is called a \textit{noncommutative measure space.}
 \item Let $(\F, \Phi)$ be a noncommutative measure space such that $\F$ is an $\A$-algebra, i.e., $\F$ is an algebra having an $\A$-bimodule structure consistent with $\F$'s own multiplication. Then the tuple $(\A,\phi,\F,\Phi)$ is called a \textit{noncommutative probability space of type B prime} (abbreviated to \textit{ncps of type B\,${}^\prime$}). 
 \end{enumerate}
\end{definition}
As in the case of ordinary ncps, we may put more assumptions on $\F$ and $\Phi$, e.g., $\Phi$ is positive if  $\F$ is a $\ast$-algebra, or $\Phi$ is tracial. In this paper, however, we do not need such assumptions.   

A typical situation is that 
$\A$ is the $\ast$-algebra $\M_N(\C)$ of matrices of size $N$, $\F$ is the $\ast$-algebra of matrices $\M_N(\C)$ of uniformly bounded rank, $\phi$ is the normalized trace and $\Phi$ is the nonnormalized trace on $\M_N(\C)$. In this case, of course $\F$ is naturally a subalgebra (even an ideal) of $\A$, but in the present paper it is very important to treat $\A$ and $\F$ as separate algebras. By doing so we can distinguish main part and perturbation part and then we can define appropriate two linear functionals, see Definition \ref{def:typeB_NCPS}.  

We define two notions of independence.

\begin{definition}\label{def:trivial} Let $(\F,\Phi)$ be a noncommutative measure space. 
  Subalgebras $(\F_i)_{i \in I}$  of $\F$ 
 are said to be \textit{trivially independent} with respect to $\Phi$ if $\Phi(f_1f_2 \cdots f_n)=0$ holds for every $n\ge2, \i  \in \I^{(n)}$ and every $(f_1, f_2,\dots, f_n) \in \F_\i$.
\end{definition}

\begin{remark}
The notion of trivial independence already appeared in the work of Ben Ghorbal and Sch\"urmann \cite{BGS} under the name of ``degenerate product'', which  is an obvious example of associative universal product. 
\end{remark}

\begin{definition}\label{def:cm} 
Let $(\A,\phi,\F,\Phi)$ be a  ncps of type B${}^\prime$, 
 $\A_1$ a subalgebra of $\A$ containing $1_\A$ and $\F_1$ a subalgebra of $\F$.
  The pair $(\A_1, \F_1)$ is said to be \textit{cyclic-antimonotone independent} if
  \begin{equation} \label{eq:cyclic}
    \Phi(a_0 f_1 a_1 f_2 \cdots a_{n-1} f_n a_n) = \varphi(a_0 a_n) \left[\prod_{1 \le i \le n-1} \phi(a_i)\right]   \Phi(f_1f_2 \cdots f_n)
  \end{equation}
  for $n \in \N$, $a_0, \dots, a_n \in \A_1$, $f_1, \dots, f_n \in \F_1$. 

\end{definition}

The concept of cyclic-monotone independence was formulated by Collins, Hasebe and Sakuma \cite{CHS} abstracting calculations by Shlyakhtenko \cite{S}. We have chosen the term ``cyclic-antimonotone'' because we prefer to write $(\A_1, \F_1)$ instead of $(\F_1, \A_1)$, the latter of which would deserve to be called cyclic-monotone independence, cf.~Definition \ref{defi:monotone}. Other models of cyclic-(anti)monotone independence are provided by Collins, Leid and Sakuma \cite{CLS} and Arizmendi, Hasebe and Lehner \cite{AHL}. Notably, the definition of cyclic-monotone independence in \cite{AHL} was given for any number of subalgebras, not limited to two subalgebras. However, in the present paper we only discuss cyclic-antimonotone independence for a pair of subalgebras.

\begin{remark}\label{rem:cm}
  Some technical comments on cyclic-antimonotone independence are worth noting here.
  \begin{enumerate}[label=\rm(\alph*)]
  \item\label{item:difference} Our definition is more or less the same as that of \cite[Definition 2.15]{CDG}, but is different from  \cite[Definition 7.2]{AHL} because we select the factor $\phi(a_0 a_n)$ instead of $\phi(a_na_0)$ in \eqref{eq:cyclic}.  Of course, this does not matter if $\phi$ is tracial, which is assumed in \cite{CHS,CLS}. However, in the operator model in \cite[Example 7.1]{AHL} $\phi$ can be nontracial although $\Phi$ is tracial.

    \item\label{rem:item:trivial}
   Except for trivial cases, if $\F_1$ is closed with respect to the action of $\A_1$ then the pair $(\A_1,\F_1)$ cannot be cyclic-antimonotone independent. Assuming to the contrary that $(\A_1,\F_1)$ is cyclic-antimonotone independent and $\A_1 \F_1 \A_1 \subseteq \F_1$, we deduce for $a_0, a_1 \in \A_1$ and $f \in \F_1$ that 
$  \Phi(a_0fa_1) = \phi(a_0 a_1)\Phi(f); 
$ 
    on the other hand, due to the fact $a_0 f \in \F_1$, 
    \[
      \Phi(a_0 f a_1) = \Phi((a_0 f) a_1)= \phi(a_1)\Phi(a_0 f) = \phi(a_0)\phi(a_1)\Phi(f).
    \]
    Therefore, either $\Restr{\varphi}{ \A_1}$ is a homomorphism or $\Restr{\Phi}{\F_1}$ is the zero map. 
    
    \item Except for trivial cases, if $(\A_1,\F_1)$ is cyclic-antimonotone independent and $\Phi$ is tracial on $\F_1$ then $\phi$ is tracial on $\A_1$. 
    For example, if there exists $f \in \F_1$ with $\Phi(f^2)\ne0$ then for every $a_1,a_2 \in \A_1$ we have
      \begin{equation*}
        \phi(a_1 a_2) \Phi(f^2) = \Phi(a_1 f^2 a_2) = \Phi((f a_2) (a_1 f)) =  \phi(a_2 a_1) \Phi(f^2),
      \end{equation*}
   and hence 
     $ 
        \phi(a_1 a_2) =\phi(a_2 a_1).
      $   
  Note that this remark does not apply to the definition of cyclic-monotone independence in \cite{AHL}, see also \ref{item:difference} above. 
  \end{enumerate}
\end{remark}

\subsection{The infinitesimal ncps associated with ncps of type B${}^\prime$} \label{sec:typeB_NCPS}

To describe a limiting behavior of random matrices with perturbation, we formulate a special type of infinitesimal ncps.

\begin{definition} \label{def:typeB_NCPS} 
Let $(\A,\phi,\F,\Phi)$ be a ncps of type B${}^\prime$. 
\begin{enumerate}[label=\rm(\roman*)]
 \item\label{item:bprime1} We construct the unital algebra $\Alg{\A}{\F}:=\A \oplus \F$ with multiplication 
  \begin{equation}\label{eq:multiplication_b'}
    (a_1,f_1) (a_1,f_2) := (a_1 a_2, a_1 f_2 + f_1 a_2 + f_1 f_2) \quad  \text{for} \quad (a_1,f_1), (a_2, f_2) \in \A \oplus \F. 
  \end{equation}
 Via the obvious embeddings of $\A$ and $\F$ into $\Alg{\A}{\F}$, we often regard $\A$ and $\F$ as subalgebras of $\Alg{\A}{\F}$.
 With this convention, we may write $a+f$ instead of $(a,f) \in \Alg{\A}{\F}$ and call $a$ the main term and $f$ the perturbation term of $a+f$.  It is clear that the unit $1_\A$ of $\A$ is also the unit of $\Alg{\A}{\F}$.
 In this paper $\Alg{\A}{\F}$ is often denoted by $\B$ and its elements are denoted by $b=a+f$.

 \item With a slight abuse of notation, we define 
  \begin{equation}\label{eq:notation_phi}
  \varphi(a+f) := \varphi(a)\quad \text{ and }\quad  \varphi'(a+f) := \Phi(f) \quad \text{for}\quad  a + f \in \Alg{\A}{\F}. 
  \end{equation} 
  We call $(\Alg{\A}{\F}, \phi, \phi')$ the \textit{infinitesimal ncps associated with $(\A,\phi,\F,\Phi)$}.  The definition of $\phi'$ is motivated by a random matrix model, see Remark \ref{rem:phi'}.  
\end{enumerate}
\end{definition}

The definition of ncps of type B${}^\prime$ is quite similar to that of type B by Biane, Goodman and Nica. For comparison, the definition of the latter is given below.  The main difference is \eqref{eq:multiplication_b'} and \eqref{eq:linking}: the multiplication of perturbation terms always vanishes in the type B case although it need not in the type B${}^\prime$ case. 

\begin{definition}[{Definition in \cite[Subsection 6.1]{BGN}}]
  A ncps of type B is a system $(\A, \phi, \V, f)$,  where 
  \begin{enumerate}
    \item $(\A,\phi)$ is a ncps, 
    \item $\V$ is an $\A$-bimodule and $f\colon\V \to \C$ is a linear functional. 
  \end{enumerate}
\end{definition}
A ncps of type B 
  $(\A, \phi, \V, f)$ associates the link-algebra, which is the vector space $\mathcal{M}=\A \oplus \V$  equipped with the associative product 
\begin{equation}\label{eq:linking}
  (a, \xi) \cdot (b, \eta) := (ab, a \eta + \xi b), 
\end{equation}
and the linear functionals $\E, \E'\colon \mathcal{M} \to \C$ by 
\begin{equation*}
\E(a,\xi):= \phi(a) \qquad \text{and} \qquad \E'(a,\xi):= f(\xi). 
\end{equation*}
Then $(\mathcal{M}, \E, \E')$ is an infinitesimal ncps.

\subsection{Weak B${}^\prime$-freeness and B${}^\prime$-freeness} \label{sec:B'free}

We consider a mixture of freeness and cyclic-antimonotone independence, abstracting the relationships between independent, rotationally invariant random matrices and finite-rank (or trace-class) perturbations.

In this subsection, we keep the setting of Definition \ref{def:typeB_NCPS}: $(\A, \varphi,\F,\Phi)$ denotes a ncps of type B${}^\prime$ and $(\B = \A\,\langle\F\rangle, \varphi, \varphi')$ the infinitesimal ncps associated with $(\A, \varphi,\F, \Phi)$.



\begin{definition} \label{def:indep_with_perturbation} Let $(\A_i)_{i\in I}$ be subalgebras of $\A$ containing $1_\A$ and $(\F_j)_{j\in J}$ be subalgebras of $\F$.  We say that the pair $((\A_i)_{i\in I}, (\F_j)_{j\in J})$ is \textit{weakly B\,${}^\prime$-free} if 
  \begin{enumerate}[label=\rm(F),leftmargin=1.5cm]
    \item $(\A_i)_{i \in I}$ are free with respect to $(\A,\varphi)$ and  
    \end{enumerate}
    \begin{enumerate}[label=\rm(CM),leftmargin=1.5cm]
   \item\label{item:CM} the pair $(\AAlg{\A_i: i \in I}, \AAlg{\F_j: j \in J} )$ of subalgebras generated by $\{\A_i\}_{i\in I}$ and $\{\F_j\}_{j\in J}$ is cyclic-antimonotone independent with respect to $(\A,\phi,\F,\Phi)$.
  \end{enumerate}
  If this is the case then at the level of infinitesimal ncps,   the restriction of $\varphi'$ to the subalgebra $\langle \A_i,\F_j: i\in I,j\in J\rangle\subseteq \B$ is uniquely determined by $\Restr{\varphi}{\A_i}, i\in I$ and $\Restr{\varphi'}{\AAlg{\F_j: j \in J}}$.

\end{definition}

\begin{definition} \label{def:typeBindependence}
 Let $(\A_i)_{i\in I}$ be subalgebras of $\A$ containing $1_\A$ and $(\F_j)_{j\in J}$ be subalgebras of $\F$. We say that the pair $((\A_i)_{i\in I}, (\F_j)_{j\in J})$ is \textit{B\,${}^\prime$-free}  if 
    \begin{enumerate}[label=\rm(wB${^\prime}$),leftmargin=1.3cm]
    \item $((\A_i)_{i\in I}, (\F_j)_{j\in J})$ is weakly B${}^\prime$-free and  
     \end{enumerate}
 \begin{enumerate}[label=\rm(T),leftmargin=1.3cm]
    \item $(\F_j)_{j \in J}$ are trivially independent.
  \end{enumerate}
  If this is the case then the restriction of $\varphi'$ to the subalgebra $\langle \A_i,\F_j: i\in I,j\in J\rangle\subseteq \B$ is uniquely determined by $\Restr{\varphi}{\A_i}, i\in I$ and $\Restr{\varphi'}{\F_j}, j \in J$. 
\end{definition}
We often consider the case $I=J$. In this case, we write \ $(\A_i, \F_i)_{i\in I}$ instead of $((\A_i)_{i\in I}, (\F_i)_{i\in I})$. 

\begin{definition}
    For a family of subsets $(S_i)_{i\in I}$ of $\A$ and subsets $(T_i)_{i\in I}$ of $\F$, we 
    say that the family $(S_i,T_i)_{i\in I}$ is weakly B${}^\prime$-free if $(\A_i, \F_i)_{i\in I}$ is weakly  B${}^\prime$-free, where 
    $\A_i $ is the subalgebra of $\A$ generated by $1_{\A}$ and $S_i$, and $\F_i$ is the subalgebra of $\F$ generated by $T_i$.  In the same way we also define B${}^\prime$-freeness for $(S_i, T_i)_{i\in I}$. 
\end{definition}


\subsection{Models from perturbed random matrices} \label{sec:RM}

Throughout this subsection, $k\in \N$ is fixed and  $(\X_{i,j}^{(N)})_{i \in [k], j \in J_i}$ and $(\Per_{i,j}^{(N)})_{i \in [k], j \in J_i}$ are families of deterministic matrices in $\M_N(\C)$.\footnote{There is no difficulty in generalizing the results of this subsection to general index sets $I$ instead of $[k]$. Nevertheless we prefer to choose $[k]$ for clear statements of asymptotic freeness. } Also, we assume that for each $i\in [k]$, the family $(\X_{i,j}^{(N)})_{j \in J_i}$ converges in distribution with respect to the normalized trace  $\tr_N$. This means that for any $i\in [k]$ and $j_1,j_2,\dots, j_n \in J_i, n\in \N$, the following limit exists in $\C$: 
\begin{equation} \label{eq:conv_distr}
    \lim_{N\to\infty}\tr_N(\X_{i, j_1}^{(N)}\X_{i, j_2}^{(N)} \cdots \X_{i, j_n}^{(N)}). 
\end{equation}
Later, we will also put assumptions on $(\Per_{i,j}^{(N)})_{i \in [k], j \in J_i}$ depending on random matrix models. 

To describe the limiting distributions of our random matrix models, we introduce the unital algebra
\begin{align*}
&\B := \C\langle x_{i,j}, \per_{i,j}: i \in [k], j\in J_i\rangle, 
\end{align*}
where $x_{i,j}, g_{i,j}~(i\in[k], j\in J_i)$ are noncommuting indeterminates. Let $\A$ be the subalgebra of $\B$ generated by $\{1_\B, x_{i,j}: i\in [k], j\in J_i\}$ and let $\F$ be the subalgebra of $\B$ consisting of polynomials in $\B$ such that every monomial contains at least one factor $\per_{i,j}$. We can also write $\F:= \A\langle \per_{i,j}: i \in [k], j\in J_i\rangle_0$ in the notation of Definition \ref{def:typeB'}.
Note that $\B$ can be naturally identified with $\Alg{\A}{\F}$.

\subsubsection*{Asymptotic weak B${}^\prime$-freeness}
Moreover, let $(U_i^{(N)})_{i\in [k]}$ be a family of independent Haar unitary random matrices in $\M_N(\C)$ and suppose that for any $i_1, i_2,\dots, i_n\in [k], j_1 \in J_{i_1}, \dots, j_n \in J_{i_n}$, the limit
\begin{equation} \label{eq:conv_distr2}
   \lim_{N\to\infty}\Tr_N(\Per_{i_1, j_1}^{(N)}\Per_{i_2,j_2}^{(N)} \cdots \Per_{i_n,j_n}^{(N)})
\end{equation}
exists, where $\Tr_N$ is the nonnormalized trace.  
For notational simplicity, we often omit the superscripts ${}^{(N)}$ below. 

\begin{theorem} \label{thm:asymp_wB'}  Assume conditions \eqref{eq:conv_distr} and   \eqref{eq:conv_distr2}. 
For any $P=P(x_{i,j},\per_{i,j}: i\in [k], j\in J_i) \in \B$ and $Q=Q(x_{i,j},\per_{i,j}: i\in [k],j\in J_i) \in \F$ the following limits exist:  
\begin{align}
&\phi(P) := \lim_{N\to\infty}\E\circ \tr_N(P(U_i \X_{i,j} U_i^*,  \Per_{i,j}: i\in [k], j\in J_i)), \label{eq:asymp_free}\\
&\Phi(Q) := \lim_{N\to\infty}\E\circ \Tr_N(Q(U_i \X_{i,j} U_i^*,  \Per_{i,j}: i\in [k], j\in J_i)). \label{eq:asymp_free2}
\end{align}
Convergences on the right-hand sides of \eqref{eq:asymp_free} and \eqref{eq:asymp_free2} also hold almost surely without taking the expectations. Moreover, the family $(\{x_{i,j}\}_{j\in J_i},\{\per_{i,j}\}_{j\in J_i})_{i\in [k]}$ is weakly B${}^\prime$-free in  $(\A,\Restr{\phi}{\A},\F,\Phi)$. 
\end{theorem}
\begin{proof} By \eqref{eq:conv_distr2}, for any polynomial $R(\per_{i,j}: i\in [k], j\in J_i)$ without a constant term, 
\begin{equation}\label{eq:F_null}
    \lim_{N\to\infty}\tr_N(R(\Per_{i,j})) =0. 
\end{equation}
Therefore, the family $(\{U_1 \X_{1,j} U_1^*\}_{j\in J_1},\{U_2 \X_{2,j} U_2^*\}_{j\in J_2}, \dots, \{U_k \X_{k,j} U_k^*\}_{j\in J_k}, \{\Per_{i,j}\}_{i\in [k], j\in J_i})$ is asymptotically free with respect to $\E\circ \tr_N$ (see \cite[Corollary 3.4]{C03}) and hence \eqref{eq:asymp_free} exists. Note that the monomials of $P$ containing some $\per_{i,j}$ as a factor do not contribute to the limit because of \eqref{eq:F_null}. Also, the almost sure version of \eqref{eq:asymp_free} is a consequence of almost sure asymptotic freeness \cite[Theorem 3.7 and Remark 3.10]{C03}. 
The convergences \eqref{eq:asymp_free2} and its almost sure version are exactly \cite[Theorem 4.8]{CHS} and \cite[Theorem 4.9]{CHS}, respectively. The last assertion follows from discussions and calculations above. 
\end{proof}


The previous theorem allows us to define the linear functionals $\phi\colon \B \to \C$ and $\Phi\colon \F\to\C$, thus providing a ncps of type B${}^\prime$ $(\A,\Restr{\phi}{\A},\F,\Phi)$. Because of  \eqref{eq:F_null} we have $\phi = \Restr{\phi}{\A}\oplus\,0$. This further associates the infinitesimal ncps $(\B,\phi,\phi')$. Recall that $\phi'$ is defined on $\B$ as $ \phi'=0\oplus \Phi$. 

\begin{remark} \label{rem:phi'}
The definition of $\phi'$ can also be given by the limit  
\[
\phi'(P)= \lim_{N\to\infty}\Tr_N[P(U_i \X_{i,j} U_i^*,  \Per_{i,j}: i\in [k], j\in J_i) - P(U_i \X_{i,j} U_i^*, 0: i\in [k], j\in J_i)],\qquad P\in \B,    
\]
i.e., it is the limit of the nonnormalized trace computed after removing the main part (the monomials of $P$ that do not contain $g_{i,j}$'s). 
\end{remark}


The construction of infinitesimal ncps can also be done at the level of random matrices, which leads to the following reformulation of Theorem \ref{thm:asymp_wB'}. We only state the almost sure version. 
\begin{corollary}\label{cor:asymp_wB'}
Let $\B_N := \M_N(\C)\oplus \M_N(\C)$ equipped with type B{}$^\prime$ multiplication \eqref{eq:multiplication_b'}. 
Then we have for all $P=P(x_{i,j},\per_{i,j}: i\in [k], j\in J_i) \in \B$ 
\begin{align*}
&\lim_{N\to\infty}(\tr_N\oplus\, 0)(P(U_i \X_{i,j} U_i^*\oplus \,0,\,  0\oplus\Per_{i,j}: i\in [k], j\in J_i)) = \phi(P) \quad \text{a.s.,}\\
& \lim_{N\to\infty}(0\oplus \Tr_N)(P(U_i \X_{i,j} U_i^*\oplus \,0,\,  0\oplus\Per_{i,j}: i\in [k], j\in J_i)) = \phi'(P)\quad \text{a.s.}
\end{align*}
\end{corollary}
\begin{remark}
    This result can be rephrased as follows: the $k+1$ subalgebras \[
    \AAlg{I_{\M_N(\C)},U_iX_{i,j}U_i^*:j \in J_i}\oplus \{0\}~(i\in[k]),\quad \C I_{\M_N(\C)} \oplus \AAlg{G_{i,j}:j\in J_i, i\in[k]}\]
    are asymptotically infinitesimally free with respect to $(\B_N, \tr_N \oplus\,0, 0\oplus \Tr_N)$, see Theorem \ref{thm:B'free}. 
\end{remark}

\subsubsection*{Asymptotic B${}^\prime$-freeness}
Let $(U_i,V_i:i\in [k])$ be a family of independent Haar unitary random matrices in $\M_N(\C)$. We now put the  assumption (weaker than \eqref{eq:conv_distr2}) that for any $i\in [k]$ and $j_1,j_2,\dots, j_n \in J_i$, the following limit exists in $\C$: 
\begin{equation} \label{eq:conv_distr3}
   \lim_{N\to\infty}\Tr_N(\Per_{i,j_1}^{(N)}\Per_{i,j_2}^{(N)} \cdots \Per_{i,j_n}^{(N)}).
\end{equation}
\begin{theorem}\label{lem:asymp_B'} Under conditions \eqref{eq:conv_distr} and  \eqref{eq:conv_distr3}, 
for any $P=P(x_{i,j},\per_{i,j}: i\in [k], j\in J_i) \in \B$ and $Q=Q(x_{i,j},\per_{i,j}: i\in [k],j\in J_i) \in \F$ the following limits exist in $\C$:  
\begin{align}
&\phi(P) := \lim_{N\to\infty}\E\circ \tr_N(P(U_i \X_{i,j} U_i^*,  V_i\Per_{i,j}V_i^*: i\in [k], j\in J_i)), \label{eq:tilde_asymp_free}\\
&\hat{\Phi}(Q) := \lim_{N\to\infty}\E\circ \Tr_N(Q(U_i \X_{i,j} U_i^*,  V_i\Per_{i,j}V_i^*: i\in [k], j\in J_i)). \label{eq:tilde_asymp_free2}
\end{align}
Convergences on the right-hand sides of  \eqref{eq:tilde_asymp_free} and \eqref{eq:tilde_asymp_free2} still hold almost surely  without taking the expectations.  Moreover, the family $(\{x_{i,j}\}_{j\in J_i},\{\per_{i,j}\}_{j\in J_i})_{i\in [k]}$ is B${}^\prime$-free in $(\A,\Restr{\phi}{\A},\F, \hat\Phi)$. 
\end{theorem}
\begin{remark} As the proofs show, the definitions \eqref{eq:asymp_free} and \eqref{eq:tilde_asymp_free} coincide. 
\end{remark}
\begin{proof} For notational simplicity we denote $\hat{\Per}_{i,j} := V_i \Per_{i,j} V_i^*$. 
    The proofs of \eqref{eq:tilde_asymp_free} and its almost sure version follow from asymptotic freeness of the family 
    \[
    (\{U_1 \X_{1,j} U_1^*\}_{j\in J_1},  \{U_2 \X_{2,j} U_2^*\}_{j\in J_2}, \dots, \{U_k \X_{k,j} U_k^*\}_{j\in J_k},  \{\hat{\Per}_{1,j}\}_{j\in J_1}, \{\hat{\Per}_{2,j}\}_{j \in J_2}, \dots, \{\hat{\Per}_{k,j}\}_{j \in J_k}).
    \] 
    Note that the monomials of $P$ that contain at least one $\per_{i,j}$ do not contribute to the limit, i.e., $\hat{\Per}_{i,j}$ converges to zero, according to condition \eqref{eq:conv_distr3}. 
    
    For  \eqref{eq:tilde_asymp_free2}, with the help of Fubini's theorem $\E = \E_{\mathbf V}\E_{\mathbf U}$, one can first fix the Haar unitaries $V_i$ and apply \eqref{eq:asymp_free2}. Note that the limit \begin{equation}\label{eq:asymp_trivial}        \lim_{N\to\infty}\E_{\mathbf V}[\Tr_N(\hat{\Per}_{i_1,j_1}\hat{\Per}_{i_2,j_2} \cdots \hat{\Per}_{i_n,j_n})] 
    \end{equation} 
    exists. The existence of this limit is exactly our assumption  \eqref{eq:conv_distr3} when $i_1=i_2= \cdots = i_n$; otherwise the limit is zero by the proof of \cite[Proposition 4.11]{CHS}. Finally, the almost sure version of \eqref{eq:tilde_asymp_free2} follows by combining the almost sure versions of \eqref{eq:asymp_free2} and  \eqref{eq:asymp_trivial}, see \cite[Proposition 4.11]{CHS}.  

    The last assertion follows from the discussions and calculations above. 
\end{proof}

This theorem allows us to define the linear functionals $\phi\colon \B \to \C$ and $\hat{\Phi}\colon \F\to\C$, thus providing a ncps of type B${}^\prime$ $(\A,\Restr{\phi}{\A},\F,\hat{\Phi})$ and the associated infinitesimal ncps $(\B,\phi,\hat{\phi}')$. Note that $\hat\phi'$ can also be defined as the limit of trace in a similar manner to Remark \ref{rem:phi'}. The analogue of Corollary \ref{cor:asymp_wB'} can  also be formulated with respect to $(\B_N, \tr_N\oplus\, 0, 0\oplus \Tr_N)$. We omit the precise statement.



\section{B${}^\prime$-freeness and infinitesimal freeness}\label{sec:BFIF}

Throughout this section, we keep the setting of Definition \ref{def:typeB_NCPS}: $(\A, \varphi,\F,\Phi)$ is a ncps of type B${}^\prime$ and $(\B = \A\,\langle\F\rangle, \varphi, \varphi')$ is the infinitesimal ncps associated with $(\A, \varphi,\F, \Phi)$. 
We first note a characterization of cyclic-antimonotone independence in terms of infinitesimal freeness.

\begin{proposition} \label{cor:equivalence}
  Let $\A_1$ be a subalgebra of $\A$ containing $1_\A$ and $\F_1$ be a subalgebra of $\F$. 
  Then $(\A_1,\F_1)$ is cyclic-antimonotone independent in $(\A, \varphi,\F,\Phi)$ if and only if  $\A_1' := \A_1\oplus \{0_\F\}$ and $\F_1' := \C 1_\A\oplus \F_1$ are infinitesimally free in $(\B, \varphi, \varphi')$.
\end{proposition}

\begin{proof} A slight modification of the proof in \cite[Proposition 3.10]{CHS} works. More precisely, one has to replace \cite[Equation (3.24)]{CHS} with 
\[
\widetilde \phi(a_0 f_1 a_1 f_2 \cdots f_n a_n) = \widetilde \phi(f_1 f_2 \cdots f_n)\widetilde\phi(a_0a_n)\prod_{i=1}^n \widetilde \phi(a_i), \qquad a_i \in \A_1', f_i \in \F_1', 
\]
where $\widetilde \phi := \phi + \epsilon \phi'$ (remember that we are assuming $\epsilon^2=0$). When proving this, it might be more useful to use the moment-cumulant formula \eqref{moment-cumulant_tilde} and vanishing of mixed cumulants, rather than using the product formula \cite[Theorem 14.4]{NS}. The other arguments are almost the same. 
\end{proof}

\begin{corollary}\label{cor:cm} Let $a \in \A$ and $f \in \F$ be cyclic-antimonotone independent with respect to $(\phi,\Phi)$. If the distribution of $a$ with respect to $\phi$ is $\mu$ and the distribution of $f$ with respect to $\phi'$ is $\nu$ then the infinitesimal distribution of $a+f$ with respect to $(\phi,\phi')$ is given by $(\mu, \tilde \nu)$, where the distribution $\tilde\nu$ is called the \textit{cyclic-antimonotone convolution} of $\mu$ and $\nu$ and is denoted by  $\clhd{\mu}{\nu}$. It has the expression $\clhd{\mu}{\nu} = \ifconv{0}{\mu\,}{\delta_0}{\nu}$ and 
is characterized by 
\begin{equation}\label{eq:cm_convolution}
    G_{\clhd{\mu}{\nu}}(z) = G_\nu (F_\mu(z)) F_\mu'(z). 
\end{equation}

\end{corollary}
\begin{remark}
Formula \eqref{eq:cm_convolution} is essentially the same as \cite[Theorem 7.6]{AHL} in which cyclic-monotone convolution is defined without subtracting the main term.  Note that although our definition of cyclic-(anti)monotone independence is different from that of  \cite{AHL} as mentioned in Remark \ref{rem:cm}, the difference does not appear in calculations of convolution as easily seen from the proofs.      
\end{remark}
\begin{proof}
From the assumption of cyclic-antimonotone independence, $b_1:=a$ and $b_2:=f$ are infinitesimally free in $(\B,\phi,\phi')$. Combining the facts that $b_1,b_2$ have the infinitesimal distributions $(\mu, 0),(\delta_0, \nu)$, respectively, and that $\omega_1(z) = z$ and $\omega_2(z) = F_\mu(z)$ we can deduce from \eqref{eq:infinitesimal_convolution} that 
\[
G_{\clhd{\mu}{\nu}} (z) = 0 + G_\nu(\omega_2(z))\omega_2'(z) = G_\nu(F_\mu(z)) F_\mu'(z)
\]
as desired. 
\end{proof}

\begin{corollary}\label{cor:B'free0}
 Assume that $(\{a_1,a_2\},\{f_1,f_2\})$ is B${}^\prime$-free. Let $b_i = a_i + f_i \in \B, i=1,2$. If $a_i$ has distribution $\mu_i$ with respect to $\phi$ and $f_i$ has distribution $\nu_i$ with respect to $\phi'$ for $i=1,2$ then the infinitesimal distribution of $b_1+b_2$ with respect to $(\phi,\phi')$ is $(\mu_1 \boxplus \mu_2,\nu)$, where $\nu$  is characterized by 
 \[
 G_\nu(z) = G_{\nu_1 + \nu_2}(F_{\mu_1 \boxplus \mu_2}(z)) F_{\mu_1 \boxplus \mu_2}'(z). 
 \]
\end{corollary}
\begin{proof}
This follows from Corollary \ref{cor:cm} and the fact that, by the trivial independence, the distribution of $f_1+f_2$ with respect to $\phi'$ is $\nu_1+\nu_2$. 
\end{proof}

(Weak) B${}^\prime$-freeness and infinitesimal freeness are also closely related as follows. 

\begin{theorem}\label{thm:B'free}
 Let $(\A_i)_{i\in I}$ be subalgebras of $\A$ containing $1_\A$ and $(\F_j)_{j\in J}$ be subalgebras of $\F$. Let $\A_i':=\A_i \oplus \{0_\F\} \subseteq \B$ and $\F_j':= (\C 1_\A)\oplus \F_j \subseteq \B$.  
 The following are equivalent. 
 
  \begin{enumerate}[label=\rm(\arabic*),leftmargin=1cm] 
   \item\label{item:thm:B'free1} $((\A_i)_{i\in I},(\F_j)_{j\in J})$ is B${}^\prime$-free in $(\A, \varphi,\F, \Phi)$;  
   

  \item\label{item:thm:B'free3} the subalgebras $\A_i'~(i\in I),\F_j'~(j \in J)$ are infinitesimally free in $(\B,\phi,\phi')$. 
 
 \end{enumerate}
Moreover, let $\F':=\C1_\A \oplus \AAlg{\F_j: j \in J}=\AAlg{\F_j': j \in J}$. The following two conditions are equivalent. 
 \begin{enumerate}[label=\rm(\roman*),leftmargin=1cm] 
   \item\label{item:thm:wB'free1} $((\A_i)_{i\in I},(\F_j)_{j\in J})$ is weakly B${}^\prime$-free in $(\A, \varphi,\F, \Phi)$;  
   
  \item\label{item:thm:wB'free2} the subalgebras $\A_i'~(i \in I), \F' $  are infinitesimally free. 
 \end{enumerate}
    
\end{theorem}
\begin{remark}
    Theorem \ref{thm:B'free} generalizes a known relation between cyclic-(anti)monotone independence and infinitesimal freeness in \cite{S}, see also \cite[Proposition 3.10]{CHS}.  
\end{remark}

\begin{proof}
  Note first that freeness of $(\A_i)_{i\in I}$ in $(\A, \phi)$ and  trivial independence of $(\F_j)_{j\in J}$ in $(\F, \Phi)$ are equivalent to infinitesimal freeness of $(\A_i')_{i\in I}$ and of $ (\F_j')_{j\in J}$  in $ (\B, \phi, \phi')$, respectively.
  By Proposition \ref{cor:equivalence}, condition \ref{item:CM} is equivalent to the infinitesimal freeness of $\A'$ and $\F'$, where 
  $ 
  \A':= \AAlg{\A_i: i \in I}\oplus \{0_\F\} = \AAlg{\A_i': i \in I}.$ 
  Combining these facts and the associativity of infinitesimal freeness, we conclude the equivalence of \ref{item:thm:B'free1} and \ref{item:thm:B'free3}. The equivalence of \ref{item:thm:wB'free1} and \ref{item:thm:wB'free2} is proved in a similar manner. 
\end{proof}

\begin{corollary}\label{cor:B'free}
In the setting of Theorem \ref{thm:B'free}, we assume that $I=J$. Let $\B_i$ be the subalgebra of $\B$ generated by $\A_i'$ and $\F_i'$. If $(\A_i,\F_i)_{i\in I}$ is B${}^\prime$-free in $(\A, \varphi,\F, \Phi)$ then $(\B_i)_{i\in I}$ are infinitesimally free in $(\B,\phi,\phi')$.  
\end{corollary}


Corollaries \ref{cor:B'free0} and \ref{cor:B'free} imply the interesting identity 
\begin{equation}\label{eq:convolution1}
    \ifconv{(\clhd{\mu_1}{\nu_1})}{\mu_1}{\mu_2}{(\clhd{\mu_2}{\nu_2})} =  \clhd{(\mu_1 \boxplus \mu_2)}{(\nu_1 + \nu_2)}.
\end{equation}
The right-hand side comes from the fact that $(a_1+a_2, f_1+f_2)$ is cyclic-antimonotone. On the other hand, the left expression comes from the facts that $(a_1,f_1)$ and $(a_2,f_2)$ are infinitesimally free and that the infinitesimal distribution of $(a_i,f_i)$ is $(\mu_i, \clhd{\mu_i}{\nu_i})$.

\section{Weak B${}^\prime$-freeness, Boolean independence and conditional freeness} \label{sec:cfree}

As usual, in this section, we keep the setting of Definition \ref{def:typeB_NCPS}: $(\A, \varphi,\F,\Phi)$ is a ncps of type B${}^\prime$ and $(\B = \A\,\langle\F\rangle, \varphi, \varphi')$ is the infinitesimal ncps associated with $(\A, \varphi,\F, \Phi)$. The following notion has already appeared in Corollary \ref{cor:B'free}.

\begin{definition}\label{def:typeB'}
A \emph{type B\,${}^\prime$ subalgebra} of $\B$ is  
the subalgebra, denoted as $\Alg{\A_1}{\F_1}$,  generated by a subspace $\A_1 \oplus \F_1\subseteq \B$, where $\A_1$ is a subalgebra of $\A$ containing $1_\A$ and a subalgebra $\F_1$ of $\F$. It has the form $\Alg{\A_1}{\F_1}=\A_1 \oplus (\nuAlg{\A_1}{\F_1})$, where $\nuAlg{\A_1}{\F_1} \subseteq\F$ is the $\A_1$-subalgebra generated by $\F_1$, i.e.,
\begin{equation*}
 \nuAlg{\A_1}{\F_1} := \mathrm{span}\{ a_0 f_1 a_1 \cdots a_{n-1} f_n a_n \mid n \in \N,\; a_0, \dots, a_n \in \A_1,\; f_1, \dots, f_n \in \F_1 \}.
\end{equation*}
Monomials $a_0 f_1 a_1 \cdots a_{n-1} f_n a_n \in \nuAlg{\A_1}{\F_1}$ are often denoted by capital letters, e.g.\ by $F$. 
Note that we do not require $\F_1$ to be closed with respect to the bilateral action of $\A_1$, cf.~Remark \ref{rem:cm} \ref{rem:item:trivial}. 
\end{definition}

\subsection{The main result and its consequences}

Connections between c-freeness and infinitesimal freeness were pointed out in \cite{BS} and \cite{CDG}.
In a similar view we find a connection between our framework and c-freeness which roughly implies: 
\begin{center}
``weak B${}^\prime$-freeness + Boolean independence $\Longrightarrow$ c-freeness''.  
\end{center}
This is to be contrasted with Corollary \ref{cor:B'free} which says: 
\begin{center}
``weak B${}^\prime$-freeness + trivial independence $\Longrightarrow$ infinitesimal freeness''. 
\end{center}

Let us introduce a new linear functional $\varphi_P$ induced from $P \in \F$ satisfying $\Phi(P)\neq 0$.

\begin{definition}
  For $P \in \F$ such that $\Phi(P)\neq 0$, we define a linear functional $\varphi_P : \B \to \C$ as follows:
  \[
    \varphi_P(b) := \frac{\Phi(Pb)}{\Phi(P)}, \qquad b \in \B.
  \]
\end{definition}
Clearly, $\varphi_P(1_\B) = 1$. 
It is easy to see that if $\A$ and $\F$ are $\ast$-algebras (the two $\ast$'s should be compatible with respect to the action of $\A$ on $\F$, i.e., $(afa')^* = (a')^* f^* a^*$), $\Phi$ is tracial and positive on $\F$, and $P=p^* p$ for some $p \in \F$ with $\Phi(P)>0$, then $\phi_P$ is a state. Of course this is the case for the random matrix model in Subsection \ref{sec:RM}.

\begin{theorem}\label{thm:c-free} Let $(\B_i=\Alg{\A_i}{\F_i})_{i\in I}$ be type B${}^\prime$ subalgebras of $\B$. Let $P$ be an element of $\F$ with $\Phi(P)\ne 0$  and let $\F_P :=\{cP^n \mid n\in\N, c\in\C\}$. Assume that 
$((\A_i)_{i \in I}, (\F_i)_{i\in I \sqcup \{P\}})$ is weakly B${}^\prime$-free.  
Then $(\B_i)_{i \in I}$ are c-free with respect to $(\phi, \phi_P)$ if and only if
  $(\F_i)_{i\in I}$ are Boolean independent with respect to $\varphi_P$.
\end{theorem}

Theorem \ref{thm:c-free} has the following immediate consequence for the sum. 

\begin{corollary} \label{prop:sum_of_typeBindep}
  Let $b_j = a_j + f_j \in \B=\A\oplus\F \; (j=1,2)$ and $P \in \F$ with $\Phi(P)\ne 0$. Assume that $a_1$ and $a_2$ are free with respect to $\varphi$, the pair $(\langle a_1, a_2\rangle , \langle P, f_1, f_2\rangle)$ is cyclic-antimonotone independent, and $f_1$ and $f_2$ are Boolean independent with respect to $\varphi_P$.
 Let $\mu_j$ be the distribution of $a_j$ with respect to $\varphi$ and also $\nu_j$ be the distribution of $f_j$ with respect to $\varphi_P$ for $j= 1,2$. Then $b_1+b_2$ has the distributions $\mu_1 \boxplus \mu_2$ with respect to $\varphi$ and 
 $\cfconv{(\mu_1\lhd \nu_1)}{\mu_1}{\mu_2}{(\mu_2\lhd \nu_2)}$ with respect to $\phi_P$. 
\end{corollary}
\begin{proof}
 Let $\lambda_i$ be the distribution of $b_i$ with respect to $\varphi_P$.
 Theorem \ref{thm:c-free} implies that $b_1$ and $b_2$ are conditionally free with respect to $(\phi,\phi_P)$, so that the distribution of $b_1+ b_2$ with respect to $\phi$ is the free convolution $\mu_1 \boxplus \mu_2$ and the distribution with respect to $\phi_P$ is $\cfconv{\lambda_1}{\mu_1}{\mu_2}{\lambda_2}$. It remains to show $\lambda_i = \mu_i \lhd \nu_i$.
 Because $(a_i,f_i)$ is cyclic-antimonotone independent with respect to $(\phi,\Phi)$, it is antimonotonically independent with respect to $\phi_P$ by Proposition \ref{Prop:cfreeforAF}; therefore the distribution of $b_i = a_i+ f_i$ is the monotone convolution of $\nu_i$ and $\mu_i$ (one can check easily that the distributions of $a_i$ with respect to $\phi_P$ and $\phi$ are the same). 
\end{proof}

According to \eqref{eq:c-free_convolution} the following formula holds: 
\begin{equation} \label{eq:c-free_convolution2}
F_{\cfconv{(\mu_1\lhd \nu_1)}{\mu_1\,\,}{\mu_2}{(\mu_2\lhd \nu_2)}}  = F_{\nu_1} \circ F_{\mu_1 \boxplus \mu_2} + F_{\nu_2} \circ F_{\mu_1 \boxplus \mu_2} - F_{\mu_1 \boxplus \mu_2}.  
\end{equation}
Combining \eqref{eq:Boolean_convolution}, \eqref{eq:monotone_convolution} and  \eqref{eq:c-free_convolution2} 
we can get another expression: 
\begin{equation}\label{eq:convolution2}
    \cfconv{(\mu_1\lhd \nu_1)}{\mu_1}{\mu_2}{(\mu_2\lhd \nu_2)} =  (\mu_1 \boxplus \mu_2) \lhd (\nu_1 \uplus \nu_2).
\end{equation}
The latter expression can also be understood from $b_1 + b_2 = (a_1 + a_2, f_1 + f_2)$ and the following facts: $a_1+a_2$ has the distribution $\mu_1 \boxplus \mu_2$ with respect to $\phi$,  $f_1+f_2$ has the distribution $\nu_1 \uplus \nu_2$ with respect to $\phi_P$, and the fact that $(a_1+a_2, b_1 + b_2)$ is cyclic-antimonotone independent with respect to $(\phi, \Phi)$ and hence is monotonically independent with respect to $\phi_P$. Finally, it is worth noting the resemblance of \eqref{eq:convolution1} and \eqref{eq:convolution2}.

\begin{example}
  \begin{enumerate}
    \item Let $f_i$ be trace class operators on a Hilbert space $H$ and $\Phi := \Tr_H$ be the canonical trace. Let $P$ be a rank one projection satisfying $f_j P = \lambda_j P, \lambda_j \in \C$ for $j = 1,2$, that is, $P$ is a rank one projection onto a simultaneous eigenvector space for $f_1$ and $f_2$. 
          Then $f_1$ and $f_2$ are Boolean (also tensor) independent with respect to $\varphi_P$. 
   \item A canonical model for Boolean independence is as follows: take Hilbert spaces $H_i$ with unit vector $\xi_i$, $i=1,2$, and take $T_i \in B(H_i)$ and define $f_1 := T_1 \otimes P_2$ and $f_2:=P_1 \otimes T_2$, where $P_i$ is the orthogonal projection $H_i\to \C\xi_i$. Then $f_1,f_2 \in B(H_1\otimes H_2)$ are Boolean independent with respect to the vacuum state corresponding to $\xi_1 \otimes \xi_2$. 
    \item If $\F_1$ and $\F_2$ are trivially independent and $P \in \F_1$, then $\F_1$ and $\F_2$ are Boolean independent. 
  \end{enumerate}
\end{example}

\subsection{Proof of Theorem \ref{thm:c-free}}
\label{subsec:CFI}

We prove Theorem \ref{thm:c-free} in two ways. The first proof is conceptual and it mainly depends on the associativity of c-freeness and the following observations.
We thank the anonymous referee for suggesting this proof.

\begin{observation} \label{Prop:cfreeforAF} For a subalgebra $\A_1 \subset\A$ that contains $1_\A$ and a subalgebra $\F_1\subset \F$ we use the notation $\A_1' := \A_1\oplus\{0_\F\} \subset \B$ and $\F_1':= (\C 1_\A)\oplus \F_1 \subset \B$ below. 
\begin{enumerate}[label=(\roman*), leftmargin=1cm]
    \item\label{obs1} Let $(\A_i)_{i\in I}$ be subalgebras of $\A$ containing $1_\A$. If $P$ is cyclic-antimonotone independent from $\langle \A_i\mid i \in I \rangle$, then $\varphi_P = \varphi$ on $\langle \A_i'\mid i\in I \rangle$. Hence, in such a situation, 
the subalgebras $(\A_i)_{i\in I}$ are free in $(\A, \phi)$ if and only if $(\A_i')_{i\in I}$ are c-free in $(\B, \varphi, \varphi_P)$.

\item\label{obs2} Let $(\F_i)_{i\in I}$ be subalgebras of $\F$. The c-freeness of $(\F_i')_{i\in I}$ with respect to $(\varphi, \varphi_P)$ is equivalent to the  Boolean independence of $(\F_i)_{i\in I}$ with respect to $\varphi_P$ because $\varphi = 0$ on $\F$ by definition.

\item\label{obs3}  Let $\A_1 \subset \A$ and $\F_1 \subset \F$ be subalgebras.
      If $(\A_1, \F_1)$ is cyclic-antimonotone independent and $P\in\F_1$ (with $\Phi(P)\neq 0$)  with respect to $(\varphi, \Phi)$ then $(\A_1', \F_1')$ is antimonotone independent with respect to $\varphi_P$, which is equivalent to that $\A_1'$ and $\F_1'$ are c-free with respect to $(\varphi, \varphi_P)$.
\end{enumerate}
    
\end{observation} 

In the last observation, the fact that cyclic-antimonotone independence implies the antimonotone independence was essentially proved by C\'ebron, Dahlqvist and Gabriel, see \cite[Proposition 2.15]{CDG}. 
The equivalence of antimonotone independence and c-freeness was observed by Franz \cite{Uwe05}.


\begin{proof}[First proof of Theorem \ref{thm:c-free}]

Let $\A' := \langle \A_i' : i\in\I \rangle \subset \B$ and $\F' := \langle \F_i' : i\in\I \rangle \subset \B$. 
From the assumption of weak B$'$-freeness, the algebras $\A'$ and $\F'$ are c-free with respect to $(\varphi, \varphi_P)$ by Observation \ref{Prop:cfreeforAF} \ref{obs3}.
Besides, the weak B$'$-freeness implies that the algebras $(\A_i)_{i\in I}$ are free with respect to $\varphi$, so that the algebras $(\A_i')_{i\in I}$ are c-free with respect to $(\varphi, \varphi_P)$ by Observation \ref{Prop:cfreeforAF} \ref{obs1}.

By the associativity of c-freeness and the facts discussed above, the c-freeness of $(\B_i)_{i\in\I}$ is equivalent to the c-freeness of $(\F_i')_{i \in I}$, the latter of which is equivalent to the Boolean independence of $(\F_i)_{i\in I}$ with respect to $\varphi_P$ by Observation \ref{Prop:cfreeforAF} \ref{obs2}. 
\end{proof}

The second proof is straightforward and it is done simply by moment calculations.

\begin{lemma}\label{lem:c-free}
  Let $(\B_i=\Alg{\A_i}{\F_i})_{i\in I}$ be type B${}^\prime$ subalgebras of $\B$. Let $P$ be an element of $\F$ with $\Phi(P)\ne 0$  and let $\F_P :=\{cP^n \mid n\in\N, c\in\C\}$. Assume that 
$((\A_i)_{i \in I}, (\F_i)_{i\in I \sqcup \{P\}})$ is weakly B${}^\prime$-free. Then
  for any $n \in \N$, $(i_1,i_2,\dots, i_n)\in \I^{(n)}$, $b_j = a_j + F_j \in \B_{i_j}=\A_i \oplus (\nuAlg{\A_i}{\F_i})$ with $\varphi(b_j) = 0 \; (1 \le j \le n)$, one has
  \[
    \varphi_P(b_1 b_2 \cdots b_n) = \varphi_P(F_1 F_2\cdots F_n).
  \]
\end{lemma}

\begin{proof} In the expansion 
\[
\Phi(Pb_1b_2\cdots b_n) = \sum_{c_j \in \{a_j,F_j\}} \Phi(P c_1 c_2 \cdots c_n), 
\]
cyclic-antimonotone independence and freeness imply that the term $\Phi(P c_1 c_2 \cdots c_n)$ vanishes as soon as there exists $j$ such that $c_j=a_j$. This yields 
  $ 
    \Phi_P(b_1b_2 \cdots b_n) = \Phi_P(F_1 F_2\cdots F_n)
  $
  as desired. 
\end{proof}

\begin{proof}[Second proof of Theorem \ref{thm:c-free}] 
   It is easy to see that the latter condition is necessary. We henceforth assume that $(\F_i)_{i\in I}$ are Boolean independent with respect to $\varphi_P$. Take any $n \in \N$, $(i_1,i_2,\dots, i_n)\in \I^{(n)}$, $b_j = a_j + F_j \in \B_{i_j}= \A_{i_j} \oplus (\nuAlg{\A_{i_j}}{\F_{i_j}})$ with $\varphi(b_j) = 0 \; (1 \le j \le n)$. 
  The goal is to establish $\phi_P(b_1b_2\cdots b_n) = \phi_P(b_1) \phi_P(b_2)\cdots \phi_P(b_n)$. By Lemma \ref{lem:c-free} this is equivalent to
  \begin{equation}\label{eq:Boolean}
    \varphi_P(F_1 F_2\cdots F_n) = \varphi_P(F_1)\varphi_P(F_2) \cdots \varphi_P(F_n), 
  \end{equation}
 i.e., $(\nuAlg{\A_{i}}{\F_{i}})_{i \in I}$ are Boolean independent with respect to $\varphi_P$. 
 
  By linearity of $\varphi_P$, we may suppose without loss of generality that every $F_j$ takes the monomial form $F_j = a^{(j)}_0 f^{(j)}_1 a^{(j)}_1 f^{(j)}_2 \cdots f^{(j)}_{n_j} a^{(j)}_{n_j} \; (n_j \in \N,\; a^{(j)}_k \in \A_{i_j},\;  f^{(j)}_k \in \F_{i_j},\; 1 \le k \le n_j)$ for $1 \le j \le n$.
  By cyclic-antimonotone independence, 
  \[
    \varphi_P(F_j)= \frac{\Phi(P a^{(j)}_0 f^{(j)}_1 a^{(j)}_1 f^{(j)}_2 \cdots f^{(j)}_{n_j} a^{(j)}_{n_j})}{\Phi(P)}= \phi_P\left(\orderedprod_{1 \le k \le n_j} f^{(j)}_k\right) \prod_{0 \le k \le n_j}\varphi(a^{(j)}_k), 
  \]
  where $\orderedprod_{1 \le k \le n_j} f^{(j)}_k$ is the ordered product $f^{(j)}_1 f^{(j)}_2 \cdots f^{(j)}_{n_j}$. 
  On the other hand, by cyclic-antimonotone independence and freeness, 
  \[
    \varphi_P(F_1 \cdots F_n) = \varphi_P\left(\orderedprod_{1\le j \le n} \;  \orderedprod_{1\le k \le n_j} f^{(j)}_k\right)
    \prod_{1 \le j \le n,\; 0 \le k \le n_j}\varphi(a^{(j)}_k). 
  \]
Because $(\F_i)_{i\in I}$ are Boolean independent with respect to $\varphi_P$,
  \[
  \varphi_P\left(\orderedprod_{1\le j \le n} \;  \orderedprod_{1\le k\le n_j} f^{(j)}_k\right) =  {\prod_{1\le j \le n}}\varphi_P\left(\orderedprod_{1\le k \le n_j} f^{(j)}_k\right).
  \]
Combining the above calculations completes \eqref{eq:Boolean}. 
\end{proof}

\section{Principal minor of unitarily invariant random matrices}
\label{sec:random_matrix}
As an application of free probability of type B$'$, we present several results on principal minors of rotationally invariant random matrices combining the type B$'$ setting and F\'evrier and Nica's infinitesimal analysis of free compressions \cite[Section 5]{FN}.

\subsection{A multivariate inverse Markov--Krein transform} 

Before going to principal minors of random matrices, we give  results on an inverse Markov--Krein transform in an abstract  setting. Throughtout this subsection, let $(\A, \varphi,\F,\Phi)$ be a ncps of type B${}^\prime$ and $(\B = \A\,\langle\F\rangle, \varphi, \varphi')$ be the associated infinitesimal ncps. 
We fix $q \in \F$ such that $q^2=q, \Phi(q)=1$ and $(\A,\{q\})$ is cyclic-antimonotone independent, and we set $p=1_\B-q \in \B$. For $a\in \A$ we introduce the notation $\tilde a := pa p \in \B$. 

Note that $q,\phi,\Phi$ are generalizations of the large $N$ limit of $\diag(0,0,\dots, 0,1) \in \M_N(\C)$, $\tr_N,\Tr_N$, respectively, see Subsection \ref{subsec:model}. We, however, do not assume the traciality of $\phi,\Phi$ in the present subsection.

\begin{theorem} \label{theorem:main}  For any $a_1,a_2,\dots, a_n \in \A$,  the following formula holds:    
  \begin{equation*}   \varphi'(\tilde{a}_{1}\tilde{a}_{2}\cdots\tilde{a}_{n}) = - \sum_{\pi \in \NC(n)} \abs{\Kr(\pi)} \kappa_\pi[a_{1},a_{2},\dots, a_{n}]. 
  \end{equation*}
  
\end{theorem}
  This is a multivariate extension of \cite[Theorem 1.3]{FH} and can be proven in three ways. The first proof, which only works in the limited setting of Subsection \ref{subsec:model}, is just to follow the lines of \cite[Theorem 1.3]{FH}. This proof is omitted in this paper since it is almost the same as \cite[Theorem 1.3]{FH} and is lengthy (one needs Weingarten calculus).
  The second proof is rather direct and is included in Appendix \ref{sec:combin} for interested readers.
  The last proof is most noteworthy and is provided below; it connects our setting of type B${}^\prime$ and infinitesimal free compressions.

\begin{remark}[Multivariate inverse Markov--Krein transform]  \label{rem:MK}
Let $\mu\colon \C[x]\to\C$ be the distribution of $a \in\A$ with respect to $\varphi$. The \emph{inverse Markov--Krein transform} of $\mu$ is the distribution $\tau\colon \C[x]\to\C$ determined by 
\begin{equation*}
 \frac{{\rm d}}{{\rm d}z}G_\mu(z) = -G_\tau(z)G_\mu(z),  
\end{equation*} 
see \cite{K}. One can easily see that $\tau(1)=1$. In the special case of $a_{1}=a_{2}=\cdots = a_{n} =:a$, \cite[Theorem 1.3]{FH} implies 
\begin{equation}\label{eq:MK}
-\phi'(\tilde a^n)=\sum_{\pi \in \NC(n)} \abs{\Kr(\pi)} \kappa_\pi[a,a,\dots, a] = \tau(x^n),  \qquad n\ge1. 
\end{equation}
Therefore, the multilinear functionals $(M_n^\phi\colon \A^n\to\C)_{n\ge1}$ defined by 
\begin{equation}\label{eq:MK3}
M_n^\phi[a_1,a_2,\dots,a_n] := -\varphi'(\tilde{a}_{1}\tilde{a}_{2}\cdots\tilde{a}_{n}) = \sum_{\pi \in \NC(n)} \abs{\Kr(\pi)} \kappa_\pi[a_1,a_2,\dots, a_n]
\end{equation} can be interpreted as a {\sl multivariate inverse Markov--Krein transform}. 

Note that essentially the same generalization was independently proposed by Arizmendi, C\'ebron and Gilliers \cite[Equation (13)]{ACG}. 
They also introduce a multivariate generalization of Kerov's operator model \cite{K} that coincides with our $-\varphi'(\tilde{a}_{1}\tilde{a}_{2}\cdots\tilde{a}_{n})$ in a special case. The basic idea is as follows. Let $H$ be a Hilbert space and $\xi \in H$ be a unit vector. Let $Q_\xi\colon H\to\C\xi$ be the orthogonal projection and let $P_\xi:=I-Q_\xi$. Then the multilinear functional on $B(H)^n$
\begin{equation} \label{eq:MK2}
(A_1,A_2,\dots, A_n)\mapsto \Tr_H[A_1 A_2 \cdots A_n - P_\xi A_1 P_\xi A_2\cdots P_\xi A_n P_\xi ]
\end{equation}
is called the inverse Markov--Krein transform in \cite{ACG}. The pair $(B(H),\{Q_\xi\})$ is cyclic-antimonotone independent with respect to $(\phi_\xi, \Tr_H)$,   
 where $\phi_\xi$ is the vacuum state $\langle \cdot~\xi,\xi\rangle$. 
Because of the common cyclic-antimonotone independence, our definition \eqref{eq:MK3} coincides with \eqref{eq:MK2} when $\A = B(H)$, $\phi= \phi_\xi$ and $\Phi=\Tr_H$.  
\end{remark}

In free probability, an element of the form $r a r$, where $r$ is a projection free from $a$, is called the \emph{free compression} of $a$ and is useful in some contexts, see e.g.\ \cite[Section 14]{NS}. 
Free compressions in infinitesimal ncps were studied by F\'evrier and Nica in \cite[Section 5]{FN}. We can directly apply their results to our setting of type B${}^\prime$ to obtain a nontrivial relation between (multivariate) inverse Markov--Krein transform and infinitesimal free compression, which will complete the proof of Theorem \ref{theorem:main} and yield more results.

To apply the framework of F\'evrier and Nica, observe the critical fact that $p$ is infinitesimally free from $\A$ in $(\B,\phi,\phi')$, which is a direct consequence of Proposition \ref{cor:equivalence}. Then let us consider the subalgebra $p \B p \subseteq \B$ with internal unit $p$ and two linear functionals $\psi,\, \psi': p \B p \to \C$ defined as 
\begin{align}
  \psi &= \phi, \label{eq:psi}\\
  \psi'&= \phi + \phi'. \label{eq:psi'}
\end{align}
Since $\varphi'(q)=1$ and $\varphi'(1)=0$, 
\[
\psi'(p) = \phi(p) + \phi'(1-q) = 1 - 1 =0.
\]
Note that this definition coincides with \cite[Definition 5.2]{FN} for $\alpha=1, \alpha'=-1$. Moreover, 
let $\ukappa_n'$ be the $n$-th infinitesimal cumulant with respect to $(\psi, \psi')$. 

\begin{proof}[Proof of Theorem \ref{theorem:main}]

 Recall that, in the type B$'$  setting, $\kappa_n'[a_1,a_2,\dots, a_n]=0$ because $\phi'$ vanishes on $\A$.  
Since $p$ is infinitesimally free from $\A$ in $(\B,\phi,\phi')$, applying \cite[Formula (5.7)]{FN}
\footnote{The original formula is erroneous; the correct one is 
\begin{equation*}
    \ukappa'_n[px_1p, px_2p,\dots, px_np] = \frac{(n-1)\alpha'}{\alpha^2}\kappa_n[\alpha x_1, \alpha x_2,\dots, \alpha x_n] + \frac{1}{\alpha}\kappa_n'[\alpha x_1, \alpha x_2, \dots, \alpha x_n], \qquad n\in \N, 
  \end{equation*}
  which can be easily obtained by calculating the coefficient of $\epsilon$ in \cite[Formula (5.8)]{FN}.} 
yields  
  \begin{equation}\label{eq:sky}
    \ukappa_n'[\tilde{a}_1,\tilde{a}_2, \dots, \tilde{a}_n] = (1-n) \kappa_n[a_1,a_2,\dots,a_n]. 
  \end{equation}
This formula, combined with the moment-cumulant formula, readily yields 
  \begin{equation}
    \begin{split}
      \psi'(\tilde{a}_1 \tilde{a}_2 \cdots \tilde{a}_n) &= \sum_{\pi \in \NC(n)} \ukappa'_\pi[\tilde{a}_1, \tilde{a}_2, \cdots \tilde{a}_n] \\
      &= \sum_{\pi \in \NC(n)} (\abs{\pi}-n) \kappa_\pi[a_1, a_2,\dots, a_n]
    \end{split}
  \end{equation}
  and hence, by the moment-cumulant formula and $\abs{\pi} + \abs{\Kr(\pi)} = n + 1$, 
  \begin{equation*}
    \begin{split}
      \varphi'(\tilde{a}_1 \tilde{a}_2 \cdots \tilde{a}_n) &= \psi'(\tilde{a}_1 \tilde{a}_2 \cdots \tilde{a}_n) - \phi(\tilde{a}_1 \tilde{a}_2 \cdots \tilde{a}_n)\\
      &= \sum_{\pi \in \NC(n)} (\abs{\pi}-n-1)\kappa_\pi[a_1,a_2, \dots, a_n]\\
      &= -\sum_{\pi \in \NC(n)}  \abs{\Kr(\pi)}\kappa_\pi[a_1,a_2, \dots, a_n]. 
    \end{split}
  \end{equation*}
\end{proof}

\begin{proposition}\label{prop:free_ifree}
  Let $(\A_i)_{i\in I}$ be free subalgebras in $ (\A, \phi)$. 
  Then $(p\A_ip)_{i\in I}$ are infinitesimally free in $(p \B p, \psi, \psi')$.
\end{proposition}
\begin{proof}
This is a direct consequence of Equation \eqref{eq:sky} and the fact that infinitesimal freeness can be characterized by vanishing of free cumulants and infinitesimally free cumulants \cite[Proposition 4.7]{FN}. 
\end{proof}


We point out several interesting connections between infinitesimal free convolution and the inverse Markov--Krein transform. 

\begin{proposition} \label{prop:MK}
  Let $a\in\A$. Let $\mu$ be the distribution of $a$ with respect to $\phi$ and $\tau$ the inverse Markov--Krein transform of $\mu$.
  Then the infinitesimal distribution of $\tilde{a} := pap$ with respect to $(p \B p,\psi, \psi')$ is $(\mu,\, \mu - \tau)$.
\end{proposition}
\begin{proof} By definition, $ \phi(\tilde a^n) = \phi(a^n)= \mu(x^n)$ for all $n\ge0$. From Remark \ref{rem:MK}, if $a$ has distribution $\mu$ with respect to $\phi$, then $\psi'(\tilde a^0)=\psi'(p)=0 = \mu(1) - \tau(1)$ and 
\[
\psi'(\tilde a^n)=\phi(\tilde a^n)+ \phi'(\tilde a^n) = \mu(x^n) -\tau(x^n),\qquad n\ge1
\]
as desired. 
\end{proof}


\begin{proposition} \label{prop:additive}
Let $\mu_i~(i=1,2)$ be distributions with $\mu_i(1)=1$ and let $\mu:= \mu_1 \boxplus \mu_2$ be their free convolution. Let  $\tau_i, \tau$ be the inverse Markov--Krein transforms of $\mu_i,\mu$~($i=1,2$), respectively. Then 
\[
(\mu_1,\mu_1-\tau_1) \boxplus (\mu_2,\mu_2 - \tau_2) = (\mu, \mu-\tau).  
\]
\end{proposition}
\begin{proof} Let $a_1,a_2\in \A$ be free elements such that $a_i$ has distribution $\mu_i$ with respect to $\phi$, $i=1,2$. By Proposition \ref{prop:free_ifree}, $pa_1 p$ and $pa_2 p$ are infinitesimally free in $(p\B p, \psi, \psi')$. Then the obvious formula $pa_1 p + p a_2p = p(a_1 + a_2)p$ and Proposition \ref{prop:MK}  yield the desired result. 
\end{proof}

Surprisingly, the multiplicative version also holds but the proof is more involved.

\begin{proposition} \label{prop:multiplicative}
  Let $\mu_i \; (i=1,2)$ be distributions with $\mu_i(1) = 1$ and let $\mu = \mu_1 \boxtimes \mu_2$ be their free multiplicative convolution.
  Let $\tau_i, \tau$ be the inverse Markov--Krein transform of $\mu_i, \mu$, respectively.
  Then
  \begin{equation}\label{eq:mult}
    (\mu_1, \mu_1 - \tau_1) \boxtimes (\mu_2, \mu_2 - \tau_2) = (\mu, \mu -\tau).
  \end{equation}
\end{proposition}
To prove Proposition \ref{prop:multiplicative}, we observe from the proof of Proposition \ref{prop:additive} that the left-hand side of \eqref{eq:mult} is the distribution of $pa_1pa_2p$ and the right-hand side is that of $pa_1a_2p$.
Thus, it is sufficient to show that the infinitesimal distributions of $pa_1pa_2p$ and $pa_1a_2p$ with respect to $(\psi, \psi')$ coincide. For this it is sufficient to show
    \begin{equation} \label{eq:moments_phi_prime}
      \phi'((pa_1a_2p)^n) = \phi'((pa_1pa_2p)^n)
    \end{equation}
    for $n \in \N$. We here use Theorem \ref{theorem:main} to  calculate separately the left and right sides of  \eqref{eq:moments_phi_prime} as 
    \begin{align*}
      \phi'((pa_1a_2p)^n) &= - \sum_{\pi \in \NC(n)} \abs{\Kr(\pi)} \fc_{\pi}[a_1a_2,a_1a_2, \dots, a_1a_2]  \\
      \phi'((pa_1pa_2p)^n) &=- \sum_{\Pi \in \NC(2n)} \abs{\Kr(\Pi)} \fc_{\Pi}[a_1, a_2,a_1,a_2, \dots, a_1, a_2].
    \end{align*}
    Thus, the proof of Proposition \ref{prop:multiplicative} is completed by the following more general combinatorial formula. 
\begin{lemma} \label{lem:core}
  Let $\A_1$,$\A_2$ be a pair of freely independent subalgebras.
  Then
  \begin{equation} \label{eq:multiplication}
    \sum_{\pi\in\NC(n)} \abs{\Kr(\pi)} \fc_\pi[a_1b_1,a_2b_2 \dots, a_nb_n] =  \sum_{\Pi\in\NC(2n)} \abs{\Kr(\Pi)} \fc_\Pi[a_1, b_1,a_2,b_2, \dots, a_n, b_n]
  \end{equation}
  for $n\in\N$, $a_j \in \A_1$, $b_j \in \A_2$ $(j \in [n])$.
\end{lemma}
\begin{remark}
The result can be extended to a family of freely independent subalgebras $\{\A_i\}_{i=1}^k$ and tuples $(a^{(1)}_1 a^{(2)}_1 \cdots a^{(k)}_1, a^{(1)}_2 a^{(2)}_2 \cdots a^{(k)}_2, \dots, a^{(1)}_n  a^{(2)}_n \cdots a^{(k)}_n)$ where $a^{(i)}_j \in \A_i$.
\end{remark}
\begin{proof}
We recall that there exists a noncrossing partition $\Kr_\pi(\sigma)$ called the relative Kreweras complement for $\sigma, \pi \in \NC(n)$ such that $\sigma \le \pi$, see \cite[Definition 18.3]{NS}.
This is just the union of the Kreweras complement of $\sigma_V$ taken in each block $V \in \pi$.
Thus, the relative Kreweras complement for $1_n$ coincides with the normal one: $\Kr_{1_n}(\pi) = \Kr(\pi)$ for $\pi \in \NC(n)$.
Also, the formula $n + \abs{\pi} = \abs{\sigma} + \abs{\Kr_{\pi}(\sigma)}$ holds.
Equivalently, 
\begin{equation} \label{eq:cardinality}
  \abs{\Kr(\pi)} = \abs{\Kr(\sigma)} + \abs{\Kr(\Kr_\pi(\sigma))} - 1.
\end{equation}
Notably, the following poset isomorphism holds: 
\begin{equation*} 
  \tau \in [\sigma, \pi] \simeq  [0_n, \Kr_\pi(\sigma)] \ni \Kr_\tau(\sigma),
\end{equation*}
where $[\sigma, \pi] = \{\tau \in \NC(n) \mid \sigma \le \tau \le \pi \}$, see \cite[Exercise 18.26 (2)]{NS}.
In particular, we have 
\begin{equation}\label{eq:poset_isom}
  \pi \in [\sigma, 1_n] \simeq  [0_n, \Kr(\sigma)] \ni \Kr_\pi(\sigma).
\end{equation}

Let $\mathbf{a}=(a_1,a_2,\dots,a_n)$ and $\mathbf{b}=(b_1,b_2,\dots,b_n)$. 
  By  \cite[Theorem 14.4]{NS}, we calculate  
  \begin{align}
      \text{the left-hand side of \eqref{eq:multiplication}} 
      &= \sum_{\pi\in\NC(n)} \abs{\Kr(\pi)} \sum_{\substack{\sigma\in\NC(n)\\ \sigma \le \pi}} \fc_{\sigma}[\mathbf{a}] \fc_{\Kr_\pi(\sigma)}[\mathbf{b}] \notag  \\
      &= \sum_{\sigma \in \NC(n)} \fc_\sigma[\mathbf{a}] \sum_{\substack{\pi\in\NC(n)\\ \sigma \le \pi}} \abs{\Kr(\pi)} \fc_{\Kr_\pi(\sigma)}[\mathbf{b}]. \label{eq:LHS}
  \end{align}
  For the right-hand side of \eqref{eq:multiplication}, thanks to the vanishing of mixed cumulants (see \cite[Theorem 11.16]{NS}), the nonzero contributions only come from the partitions $\Pi \in \NC(2n)$ that can be decomposed into a pair $(\pi_1, \pi_2) \in \NC(I_1) \times \NC(I_2)$ with $I_1 = \{1, 3, \dots, 2n-1 \}$ and $I_2 = \{ 2, 4, \dots, 2n\}$.  
  Thus,
  \begin{align}
      \text{the right-hand side of \eqref{eq:multiplication}} 
      &= \sum_{\Pi \in \NC(2n)} (2n+1 - \abs{\Pi}) \fc_\Pi[a_1,b_1,a_2,b_2,\dots,a_n,b_n] \notag \\
      &= \sum_{\substack{\pi_1, \pi_2 \in \NC(n)\\ \pi_2 \le \Kr(\pi_1)}} (2n+1 - \abs{\pi_1} - \abs{\pi_2})\fc_{\pi_1}[\mathbf{a}] \fc_{\pi_2}[\mathbf{b}]  \notag \\
      &= \sum_{\pi_1 \in \NC(n)} \fc_{\pi_1}[\mathbf{a}] \sum_{\substack{\pi_2 \in \NC(n) \\ \pi_2 \le \Kr(\pi_1)}} (\abs{\Kr(\pi_1)} + \abs{\Kr(\pi_2)} - 1) \fc_{\pi_2}[\mathbf{b}]  \label{eq:RHS1}
  \end{align}
  The poset isomorphism \eqref{eq:poset_isom} implies
  \begin{equation}
    \text{\eqref{eq:RHS1}} = \sum_{\sigma \in \NC(n)} \fc_{\sigma}[\mathbf{a}] \sum_{\substack{\pi \in \NC(n) \\ \sigma \le \pi}} (\abs{\Kr(\sigma)} + \abs{\Kr(\Kr_{\pi}(\sigma))} - 1) \fc_{\Kr_{\pi}(\sigma)}[\mathbf{b}]. \label{eq:RHS2}
  \end{equation}
  Formula \eqref{eq:cardinality} then implies
  \begin{equation*}
    \text{\eqref{eq:RHS2}} = \sum_{\sigma \in \NC(n)} \fc_{\sigma}[\mathbf{a}] \sum_{\substack{\pi \in \NC(n) \\ \sigma \le \pi}} \abs{\Kr(\pi)} \fc_{\Kr_{\pi}(\sigma)}[\mathbf{b}], 
  \end{equation*}
 which coincides with \eqref{eq:LHS} as desired. 
  \end{proof}

  Yet remarkably, the analogue of Propositions \ref{prop:additive} and \ref{prop:multiplicative} fails for (anti-)commutators.
  A counterexample is presented below, only for the anti-commutator because the commutator can be treated similarly.
  Let $a_1, a_2$ be free random variables and define  $x := \widetilde{a_1 a_2} + \widetilde{a_2a_1} = pa_1a_2p + pa_2a_1p$ and $X := \widetilde{a}_1 \widetilde{a}_2 + \widetilde{a}_2\widetilde{a}_1 =   pa_1pa_2p +pa_2pa_1p$.
  We shall prove the infinitesimal distributions of $x$ and $X$ do not coincide in general.
   Their second infinitesimal moments are given by 
  \begin{align*}
    \phi'(x^2) &=  \phi'(\widetilde{a_1 a_2}\widetilde{a_1 a_2}) + \phi'(\widetilde{a_1 a_2}\widetilde{a_2 a_1}) + \phi'(\widetilde{a_2 a_1}\widetilde{a_1 a_2}) + \phi'(\widetilde{a_2 a_1}\widetilde{a_2 a_1}), \label{eq:x^2} \\
    \phi'(X^2) &=  \phi'(\widetilde{a}_1\widetilde{a}_2\widetilde{a}_1\widetilde{a}_2) + \phi'(\widetilde{a}_1\widetilde{a}_2\widetilde{a}_2\widetilde{a}_1) + \phi'(\widetilde{a}_2\widetilde{a}_1\widetilde{a}_1\widetilde{a}_2) + \phi'(\widetilde{a}_2\widetilde{a}_1\widetilde{a}_2\widetilde{a}_1). 
  \end{align*}
 The first and fourth terms coincide between $\phi'(x^2)$ and $\phi'(X^2)$  by Equation \eqref{eq:moments_phi_prime}.
  By Theorem \ref{theorem:main}, the sum of second and third terms of $\phi'(x^2)$ is
  \[
    \begin{split}
      &\sum_{\pi \in \NC(2)} \abs{\Kr(\pi)} (\fc_{\pi}[a_1a_2, a_2a_1] + \fc_{\pi}[a_2a_1, a_1a_2])\\
      &= \fc_1[a_1a_2] \fc_1[a_2a_1] + \fc_1[a_2a_1] \fc_1[a_1a_2] + 2(\fc_2[a_1a_2, a_2a_1] + \fc_2[a_2a_1, a_1a_2])\\
      &=2\fc_1[a_1]^2\fc_1[a_2]^2 + 4 \fc_2[a_1, a_1]\fc_1[a_2]^2 + 4\fc_1[a_1]^2\fc_2[a_2, a_2] + 4\fc_2[a_1, a_1] \fc_2[a_2, a_2],
    \end{split}
  \]
  while that of $\phi'(X^2)$ is
  \[
    \begin{split}
      &\sum_{\Pi \in \NC(4)} \abs{\Kr(\Pi)} (\fc_{\Pi}[a_1, a_2, a_2, a_1] + \fc_{\Pi}[a_2, a_1, a_1, a_2])\\
      &= 2 \fc_1[a_1]^2 \fc_1[a_2]^2 + 4 \fc_2[a_1, a_1]\fc_1[a_2]^2 + 4\fc_1[a_1]^2\fc_2[a_2, a_2] + 6\fc_2[a_1, a_1] \fc_2[a_2, a_2].
    \end{split}
  \]
For example if $a_1$ and $a_2$ are the standard free semi-circular elements then $4 = \phi'(x^2) \neq \phi'(X^2) = 6$.

\subsection{Asymptotic infinitesimal freeness of principal minors}\label{subsec:model}

We work on a special case of the abstract setting in Subsection \ref{subsec:model} coming from random matrices.  
 Let $\B=\C\langle x_{i,j},q:i\in [k], j\in J_i\rangle$ be the unital polynomial algebra generated by noncommuting indeterminates $x_{i,j}~(i\in [k], j\in J_i), q$ with relation $q^2=q$, and let  
\[
\A=  \C\langle x_{i,j}: i\in[k],j \in J_i\rangle, \qquad \F= \A\langle q\rangle_0.  
\]
Note that $\B$ can be naturally identified with $\A \oplus \F$. 
Suppose that, for each $i\in[k]$, $(\X_{i,j}=\X_{i,j}^{(N)})_{j \in J_i}$ is a family of deterministic matrices in $\M_N(\C)$ which converge in distribution with respect to $\tr_N$, and $(U_i=U_i^{(N)})_{i\in[k]}$ is an independent family of Haar unitary random matrices.  
We consider the principal submatrices $(\tilde {\X}_{i,j}=\tilde{\X}_{i,j}^{(N)}:=P U_i\X_{i,j} U_i^*P)_{i\in[k],j\in J}$, where $P=P^{(N)}=\diag(1,1,\dots,1,0)$.  

We define the unital linear functional
\[
\phi\colon\B\to \C, \qquad \phi(R):= \lim_{N\to\infty} \E[\tr_N(R(U_i\X_{i,j}U_i^*,P:i\in[k],j\in J_i))]. 
\]
Note that $(\B,\phi)$ is a conventional setting of type A and $\tr_N(P^n) \to 1$ for any $n\in\N$, so that one cannot observe any difference between $(U_i\X_{i,j}U_i^*)$ and their submatrices $(\tilde{\X}_{i,j})$. 
In the framework of type B${}^\prime$, however, one can capture a difference between them. 
The matrix $Q=Q^{(N)}:=  I_{\M_N(\C)} - P^{(N)}$ obviously satisfies 
\[
\lim_{N\to\infty}\Tr_N(Q^n) =1 
\]
for any $n\in\N$ and thus by Theorem \ref{thm:asymp_wB'} one can define a linear functional $\Phi\colon \F\to\C$ by
\[
\Phi(R) := \lim_{N\to\infty}\E[\Tr_N(R(U_i \X_{i,j} U_i^*, Q:i\in[k], j\in J_i))], 
\]
where $R=R(x_{i,j}, q: i\in[k], j\in J_i)$ is a polynomial in which every monomial contains at least one $q$. 
Thus we obtain the ncps of type B${}^\prime$ $(\A,\Restr{\phi}{\A},\F,\Phi)$ and the associated infinitesimal ncps $(\B,\phi,\phi')$. In the space $\B$ we can define $p := 1_\B -q \in \B$.  

By Theorem  \ref{thm:asymp_wB'}, the family  $(\{x_{i,j}\}_{j\in J_i},\{q\})_{i\in [k]}$, which is the limit of $(\{U_i\X_{i,j} U_i^*\}_{j \in J_i}, \{Q\})_{i \in [k]}$, is weakly B${}^\prime$-free in $(\B,\phi,\phi')$. Recall that this fact can be reformulated as Corollary \ref{cor:asymp_wB'}, which leads to an asymptotic infinitesimal freeness of principal minors.

\begin{theorem}\label{thm:asymp_IF} Let $\B_N:=\M_N (\C) \oplus \M_N(\C)$ equipped with type B{}$^\prime$ multiplication \eqref{eq:multiplication_b'}.  We interpret the projection $P  = I - Q$ as $P = I \oplus (-Q) \in \B_N$. With this interpretation, we consider the subalgebra $P \B_N P\subseteq \B_N$ having internal unit $P$. 
Let $\A_i^0$ be the random sublagebra of $\B_N$ generated by $\{U_iX_{i,j}U_i^*\oplus 0: j\in J_i\}$. 
Then the random subalgebras 
$ 
\{P \A_i^0   P\}_{i\in [k]}
$ 
are almost surely asymptotically infinitesimally free in the infinitesimal ncps $(P\B_N P,\tr_N\oplus\, 0, \tr_N\oplus\Tr_N)$.   

\end{theorem}

\begin{proof} 
The desired conclusion means that the subalgebras $p\langle x_{i,j} : j \in J_i\rangle p, i\in [k],$ are infinitesimally free in $(p\B p,\psi,\psi')$, where $\psi, \psi'$ are defined in \eqref{eq:psi}, \eqref{eq:psi'}, respectively. This follows from Proposition \ref{prop:free_ifree} applied to the subalgebras $\A_i:=\langle x_{i,j} : j \in J_i\rangle$. 
\end{proof}

As a final comment, a crucial idea in Theorem \ref{thm:asymp_IF} was to regard the principal minor $\tilde X$ of a matrix $X \in \M_N(\C)$ as the element
\begin{align*}
\tilde X &= P (X \oplus 0)P = X  \oplus [-Q X  - X Q + Q X Q] 
\end{align*}
in the larger space $\B_N$. In $\B_N$ we can define a nonnormalized trace on the whole algebra $\B_N$ as $0\oplus \Tr_N$ that converges in the limit. 




\appendix

\section{Combinatorial proof of Theorem \ref{theorem:main}} \label{sec:combin}

We provide yet another straightforward proof of Theorem \ref{theorem:main} in this appendix.
The proof depends only on the following combinatorial lemma.
\begin{lemma} \label{lemma:kappa'}
  For every $\sigma \in \NC(n)$, 
  \begin{equation*}
    \sum_{\pi \le \sigma} \kappa'_\pi[p,p,\dots,p] = -\abs{\sigma}.
  \end{equation*}
\end{lemma}
\begin{proof}
\ 

\subproof{Step 1: $\sigma= 1_n$.} 
Concerning moments with respect to $\phi$ or cumulants $\kappa_\pi$, the perturbation part $q$ can be regarded as $0$ and so the element $p=1_\A -q$ is indistinguishable with $1_\A$. This implies that $\Restr{\kappa}{V}[p,p,\dots, p]$ vanishes if $V \subseteq [n]$ contains more than one element. Hence, by the Leibniz rule for $\kappa_\pi'$ (the second formula of \eqref{eq:multiplicativity_kappa}), only the noncrossing partitions $\pi$ having at most one non-singleton block contribute to the sum, i.e., the partitions $0_n$ and $\pi_V:=\{V, \{k\}\mid k \in [n]\setminus V\}$ for all $V \subseteq[n]$ with $|V|\ge2$.  

For each $l\in \N$, {the moment-cumulant formula \eqref{eq:multiplicativity_kappa} and the fact that $\phi$ vanishes on $\{q\}$} readily yield 
\[
\kappa_l'[p,p,\dots,p] = (-1)^l\kappa_l'[q,q,\dots,q] = (-1)^{l}\Phi(q) = (-1)^{l},  
\]
and hence, by the Leibniz rule, 
\begin{equation*}
    \kappa'_{0_n}[p,p,\dots,p] = -n \qquad \text{and} \qquad \kappa'_{\pi_V}[p,p,\dots,p] = (-1)^{\abs{V}}. 
  \end{equation*}
Therefore,
  \begin{equation*}
      \sum_{\pi \in \NC(n)} \kappa'_\pi[p,p,\dots,p]
      = - n + \sum_{V\subseteq [n], |V|\ge2} (-1)^{|V|} 
      = -n + \sum_{2 \le l \le n} \binom{n}{l} (-1)^l = -1.
  \end{equation*}
  
 \subproof{Step 2: general $\sigma$.}  
  Let $\sigma = \{ V_1, \dots, V_k \} \in \NC(n)$. Then $\{ \pi \in \NC(n) \mid \pi \le \sigma\}$ is naturally isomorphic to $\NC(V_1) \times \cdots \times \NC(V_k)$, so that (with $[p,p,\dots,p]$ omitted) 
  \begin{equation*}
    \begin{split}
      \sum_{\pi \le \sigma} \kappa'_\pi &= \sum_{\pi_1 \in \NC(V_1)} \dots \sum_{\pi_k \in \NC(V_k)} (\kappa_{\pi_1} \cdots \kappa_{\pi_k})'  \\
      &= \sum_{1 \le j \le k} \sum_{\pi_1 \in \NC(V_1)} \dots \sum_{\pi_k \in \NC(V_k)} \kappa'_{\pi_j} (\kappa_{\pi_1} \cdots \kappa_{\pi_{j-1}} \kappa_{\pi_{j+1}} \cdots \kappa_{\pi_k})\\
      &= \sum_{1 \le j \le k} \sum_{\pi_j \in \NC(V_j)} \kappa'_{\pi_j}\\ 
      &= -k; 
    \end{split}
  \end{equation*}
 from the second to the third line the fact that $\kappa_{\pi_j}=0$ for all $\pi_j \ne 0_{V_j}$ is used; the last computation follows from Step 1. 
\end{proof}

\begin{proof}[Proof of Theorem \ref{theorem:main}]
We can calculate as follows:
  \begin{align}
      \varphi'(\tilde{a}_1 \tilde{a}_2 \cdots \tilde{a}_n) &= \varphi'(p a_1 p a_2 p \cdots p a_n p) \notag\\
      &= \varphi'(1_\A p a_1 p a_2 p \cdots p a_n p) \notag\\
      &= \sum_{\pi \in \NC(2n+2)} \kappa'_\pi[1_\A, p, a_1, p,a_2,p, \dots, a_n, p]\notag\\
      &= \sum_{\pi \in \NC(n+1)} \kappa_\pi[1_\A, a_1, a_2,\dots, a_n] \sum_{\sigma \le \Kr(\pi)} \kappa'_\sigma[p, p, p,\dots, p] \notag\\
      &=  - \sum_{\pi \in \NC(n+1)} \abs{\Kr(\pi)} \kappa_\pi[1_\A, a_1, a_2,\dots, a_n].\label{eq:combinotorial}
  \end{align}
The fourth line is based on the fact that $\A$ is infinitesimally free from $p$ and the fifth line is based on Lemma \ref{lemma:kappa'}.
To the last summation only the subset $\{\pi \in \NC(n+1) \mid \{1\} \in \pi\}$ contributes, which is isomorphic to $\NC(n)$.
This isomorphism does not change the cardinality of the Kreweras complement, so that Theorem \ref{theorem:main} follows. 
\end{proof}

\section*{Acknowledgements}
This work was supported by JSPS Open Partnership Joint Research Projects grant no. JPJSBP120209921 and Bilateral Joint Research Projects (JSPS-MEAE-MESRI, grant no. JPJSBP120203202).
The first author was supported by the Hokkaido University Ambitious Doctoral Fellowship (Information Science and AI) and JSPS Research Fellowship for Young Scientists PD (KAKENHI Grant Number 24KJ1318).
The second author is supported by JSPS Grant-in-Aid for Transformative Research Areas (B) grant no.~23H03800JSPS, JSPS Grant-in-Aid for Young Scientists 19K14546 and JSPS Scientific Research 18H01115.
The authors are grateful to Octavio Arizmendi for discussions about the preprint \cite{ACG} and for his hospitality during the authors' stays in CIMAT where part of this work was done; to Dan Voiculescu for asking the question about the (anti-)commutator version of Propositions \ref{prop:additive} and \ref{prop:multiplicative}; to Pei-Lun Tseng for discussions about Theorem \ref{theorem:main} in the nontracial setting. Finally, the authors express sincere gratitude to an anonymous referee for helpful comments that considerably shortened some proofs and improved the readability of the paper.



\vspace{6mm}
\begin{enumerate}
  \item[]
    \hspace{-10mm} Katsunori Fujie\\
    Department of Mathematics, Hokkaido University.\\
    North 10 West 8, Kita-Ku, Sapporo 060-0810, Japan.\\
   URL: \url{https://sites.google.com/view/katsunorifujie} \\
     Current address: Department of Mathematics, Kyoto University.\\
  Kitashirakawa, Oiwake-cho, Sakyo-ku, Kyoto, 606-8502, Japan.\\
 email: fujie.katsunori.42m@st.kyoto-u.ac.jp\\
   
  \item[]
    \hspace{-10mm} Takahiro Hasebe\\
    \hspace{-10mm} Department of Mathematics, Hokkaido University.\\
    \hspace{-10mm} North 10 West 8, Kita-Ku, Sapporo 060-0810, Japan. \\
    \hspace{-10mm} email: thasebe@math.sci.hokudai.ac.jp\\
    \hspace{-10mm} URL: \url{https://www.math.sci.hokudai.ac.jp/~thasebe/}
\end{enumerate}

\end{document}